\documentclass[11pt]{article}
\usepackage{amsmath,latexsym,amssymb,amsfonts,amsbsy, amsthm}
\usepackage{graphicx}
\usepackage{ulem}
\def\suite#1#2#3{(#1_{#2})_{#2\in {#3}}}

\def\inte#1{
\displaystyle\mathop{#1\kern0pt}^\circ }

\let\pa=\partial

\def\cD{{\cal D}}

\let\grad\nabla

\def\virgp{\raise 2pt\hbox{,}}
\def\cdotpv{\raise 2pt\hbox{;}}

\def\eqdefa{\buildrel\hbox{\footnotesize def}\over =}

\def\C{\mathop{\bf C\kern 0pt}\nolimits}
\def\DD{\mathop{\bf D\kern 0pt}\nolimits}
\def\K{\mathop{\bf K\kern 0pt}\nolimits}
\def\N{\mathop{\bf N\kern 0pt}\nolimits}
\def\Q{\mathop{\bf Q\kern 0pt}\nolimits}
\def\R{\mathop{\bf R\kern 0pt}\nolimits}
\def\SS{\mathop{\bf S\kern 0pt}\nolimits}
\def\ZZ{\mathop{\bf Z\kern 0pt}\nolimits}
\def\TT{\mathop{\bf T\kern 0pt}\nolimits}

\newcommand{\Z}{{\mathbf Z}}

\def\dive{\mathop{\rm div}\nolimits}

\def\Supp{\mathop{\rm Supp}\nolimits\ }

\newcommand{\beq}{\begin{equation}}
\newcommand{\eeq}{\end{equation}}
\newcommand{\ben}{\begin{eqnarray}}
\newcommand{\een}{\end{eqnarray}}
\newcommand{\beno}{\begin{eqnarray*}}
\newcommand{\eeno}{\end{eqnarray*}}

\newtheorem{example}{Example}[section]
\newtheorem{thm}{Theorem}[section]
\newtheorem{lem}{Lemma}[section]
\newtheorem{rmk}{Remark}[section]

\newtheorem{prop}{Proposition}[section]
\newtheorem{assumption}{Assumption}[section]
\renewcommand{\theequation}{\thesection.\arabic{equation}}

\begin{document}

\title{Global well-posedness of the 2-D incompressible Navier-Stokes-Cahn-Hilliard system with singular free energy densities}

\author{
Guilong Gui \footnote{Center for Nonlinear Studies, School of Mathematics, Northwest University, Xi'an 710069, China. Email: {\tt glgui@amss.ac.cn}.} \and Zhenbang Li \footnote{Center for Nonlinear Studies, School of Mathematics, Northwest University, Xi'an 710069, China;  School of Science, Xi'an Technological University, Xi'an 710021, China. Email: {\tt lizbmath@nwu.edu.cn}.}
}

\date{}
\maketitle

\begin{abstract}
Consideration in this paper is the effects of some singular free energy densities on global well-posedness of the 2-D incompressible Navier-Stokes-Cahn-Hilliard (NS-CH) system. Due to lack of the maximum principle for the convective Cahn-Hilliard equation (as a fourth-order parabolic equation), we construct its approximate second-order parabolic equation, and use comparison principle and the basic energy estimates to separate the solution from the singular values of the singular free energy density, where the Orlicz embedding theorem plays a key role. Based on these, we prove the global well-posedness of the Cauchy problem of the 2-D NS-CH equations with periodic domains by using energy estimates and the logarithmic Sobolev inequality.
\end{abstract}


\vskip 0.2cm


\renewcommand{\theequation}{\thesection.\arabic{equation}}
\setcounter{equation}{0}
\section{Introduction}

We consider herein a diffuse interface model which describes the evolution of droplet formation and collision during flow of viscous
incompressible Newtonian fluids of the same density but different viscosity. There are two sides in this situation. From the macroscopical points of view, the fluids are immiscible. On the other hand, the model considers a partial mixing on a small length scale measured by a parameter $\kappa>0$, such as interface of two-phase fluids, oil and water for instance. The model was first discussed by P. C. Hohenberg and B. I. Halperin \cite{HH}, and then derived in the framework of rational continuum mechanics by M. E. Gurtin \cite{GME}.  This system consists of
the incompressible Navier-Stokes equations coupled with a convective Cahn-Hilliard equation, which leads to the following incompressible Navier-Stokes-Cahn-Hilliard (NS-CH for short) equations
\begin{equation}\label{1-NSCH-1}
\begin{cases}
&\partial_t u+u\cdot\nabla u-\dive(2\nu(\theta)Du)+\nabla p=-\kappa\dive(\nabla\theta\otimes\nabla\theta), \quad \forall\, (t, x) \in \mathbb{R}^+ \times \Omega,
\\
&\partial_t\theta+u\cdot\nabla\theta = m \Delta\mu, \quad \mu =\kappa^{-1}\phi(\theta)-\kappa\Delta\theta,
\\
&\dive u =0.
\end{cases}
\end{equation}
Here, $\Omega \subset \mathbb{R}^d$ ($d=2, \, 3$) is bounded domain with smooth boundary, the unknown $u$ is the velocity field of the fluid, $Du=\frac{1}{2}(\nabla u+\nabla u^T)$, the unknown $p$ is a scalar pressure function, the unknown $\theta$ is an order
parameter related to the concentration of the fluids, and the viscosity coefficient $\nu(\cdot)$ is a smooth, positive function on $\mathbb{R}$, $\kappa>0$ is a parameter related to the ¡°thickness¡± of the interfacial region, $\phi=\Phi'$ for some suitable energy density $\Phi$ specified below, and $\mu$ is the so-called chemical potential. It is assumed here that the densities of both components as well as the density of the mixture are constant and for simplicity equal to one. The extra force $-\kappa\dive(\nabla\theta\otimes\nabla\theta)$ appearing in the right-hand side of the first equation in \eqref{1-NSCH-1} can be considered as the capillary force due to the surface tension. Moreover, we also take into account the diffusion $m \Delta\mu$ in the the concentration $\theta$-equation in \eqref{1-NSCH-1}, where the mobility coefficient $m$ is a positive constant.

The system \eqref{1-NSCH-1} is usually posed by the following boundary and initial conditions
\begin{equation}\label{bc-1}
 \begin{cases}
 &u|_{\partial \Omega}=0,\,\partial_n \theta|_{\partial \Omega}=\partial_n \mu|_{\partial \Omega}=0,\\
 &u|_{t=0}=u_0, \, \theta|_{t=0}=\theta_0.
\end{cases}
\end{equation}

For the sake of simplicity, we may assume in the paper that the fluid occupies the two dimensional torus $\Omega=\mathbb{T}^2$, because, in the physical experiments, the shear is obtained by putting the mixture between two rotating cylinders whose diameters are very close (Couette-Taylor flows), curvature effects are usually neglected because of the thickness of the domain (see \cite{BFFP, GG}).

Thanks to the identity $-\kappa\dive\,(\nabla\theta\otimes\nabla\theta)=\mu\nabla\theta-\nabla(\frac{\kappa}{2}|\nabla\theta|^2+\kappa^{-1}\Phi(\theta))$,
the momentum equations in \eqref{1-NSCH-1} can be rewritten as
\begin{equation*}\label{1-6}
\partial_t u+u\cdot\nabla u-\dive(\nu(\theta)Du)+\nabla g=\mu\nabla\theta
\end{equation*}
with the total pressure $g:=p+\frac{\kappa}{2}|\nabla\theta|^2+\kappa^{-1}\Phi(\theta)$, where the term  $\mu\nabla\theta$ is known as Korteweg force, in which the chemical potential $\mu$ can be considered as the variational derivative of the following Ginzburg-Landau free energy
\begin{equation*}\label{1-7}
\mathcal{E}(\theta)=\int_{\mathbb{T}^2}(\frac{\kappa}{2}|\nabla\theta|^2+\kappa^{-1}\Phi(\theta))\,dx,
\end{equation*}
where the potential $\Phi(\theta):=\int^{\theta}_0\phi(\eta)\,d\eta$ is the Helmholtz free energy density.

We will mainly pay attention in this work to a theory for a class of physically
relevant, singular free energy densities $\Phi$. More precisely, we assume throughout the article:
\begin{assumption}
 \label{assumption-1}
 Let $\Phi(\cdot)\in C([a, b]) \cap C^{\infty}((a, b))$ such that $\phi=\Phi'$ satisfies
\begin{equation}\label{1-9}
\begin{split}
&\lim_{s\rightarrow a^+} \phi(s)=-\infty, \quad\lim_{s\rightarrow b^-} \phi(s)=+\infty,\\
& \qquad\qquad  \phi'(s) \geq -\alpha, \quad |\phi'(s)| \leq C_1 e^{C_2|\phi(s)|}+C_3, \quad \forall \, s \in (a, b)
\end{split}
\end{equation}
for some constants $\alpha, \, C_1, \, C_2, \, C_3>0$. We extend $\Phi(s)=+\infty$ if $s  \notin [a, b]$.
\end{assumption}

\begin{rmk}\label{rmk-Phi-1}
We claim that $\int_{\mathbb{T}^2}\Phi(\theta)\,dx <\infty$ implies $\theta(x) \in [a, b]$ for
almost every $x\in \mathbb{T}^2$.  Indeed, since
\begin{equation*}\label{1-9-1}
\begin{split}
\int_{\mathbb{T}^2}\Phi(\theta)\,dx =\int_{\{\theta <a\} \cup \{\theta>b\} }\Phi(\theta)\,dx +\int_{\theta\in [a, b]}\Phi(\theta)\,dx\geq \int_{\{\theta <a\} \cup \{\theta>b\} }\Phi(\theta)\,dx-C,
\end{split}
\end{equation*}
we get from $\int_{\mathbb{T}^2}\Phi(\theta)\,dx <\infty$  that
$meas(\{\theta <a\} \cup \{\theta>b\})=0$, that is, $\theta(x) \in [a, b]$ for almost every $x\in \mathbb{T}^2$.
\end{rmk}
Without loss of generality, we assume in the article that $\theta$ is just the concentration difference of both components and $[a, b] =
[-1, 1]$.

\begin{example}
 \label{example-1}
Assumption \ref{assumption-1} is motivated by the so-called regular solution
model free energy suggested by Cahn and Hilliard \cite{CH1958}:
\begin{equation}\label{1-8}
\Phi(s)=
\begin{cases}
&\frac{\alpha_0}{2}\bigg((1+s)\ln(1+s)+(1-s)\ln(1-s)\bigg)-\frac{\alpha}{2}s^2, \quad \mbox{if} \quad  s\in[-1,1],\\
& +\infty, \quad \mbox{if} \quad  s \notin[-1,1],
\end{cases}
\end{equation}
where two constants $\alpha_0$ and $\alpha$ satisfy $0<\alpha_0<\alpha$, $a=-1$, $b=1$. Here the logarithmic terms are related to the entropy of the system.
\end{example}

\begin{rmk}\label{rmk-5}
The function $\phi=\Phi'$ in \eqref{1-8} satisfies \eqref{1-9}.
In order to verify this, we need only to show
\begin{equation*}\label{2-19}
|\phi'(s)|\leq e^{\frac{2}{\alpha_0}|\phi(s)|+\frac{2\alpha}{\alpha_0}+\ln{\alpha_0}}+\alpha, \quad \forall \, s\in(-1,1),
\end{equation*}
where $\phi(s)=\frac{\alpha_0}{2}\ln \frac{1+s}{1-s}-\alpha\, s$ and $\phi'(s)=\frac{\alpha_0}{1-s^2}-\alpha$, $\forall \, s\in(-1,1)$, from \eqref{1-8}.
Indeed, we first introduce $f(s):=|\ln{\frac{1+s}{1-s}}|-\ln{\frac{1}{1-s^2}}$ for $s\in(-1,1)$. It is easy to find that $f(0)=0$ and $f(s) \geq 0$ for any $s\in (-1,1)$.
Therefore, it follows that $\forall\,  s\in(-1,1)$
\begin{align*}\label{2-22}
|\phi'(s)|\leq |\frac{\alpha_0}{1-s^2}|+\alpha\leq e^{\ln{|\frac{\alpha_0}{1-s^2}|}}+\alpha\leq e^{|\ln{\frac{1+s}{1-s}}|+\ln{\alpha_0}}+\alpha\leq e^{\frac{2}{\alpha_0}|\phi(s)|+\frac{2\alpha}{\alpha_0}+\ln{\alpha_0}}+\alpha.
\end{align*}
\end{rmk}

\begin{rmk}
  The logarithmic potential \eqref{1-8} is often replaced by a smooth double-well polynomial approximation \cite{AH1, BF, GG, GG2}, such as $\Phi(s)= \gamma_1 s^4-\gamma_2 s^2$ and $\Phi(s)= (1-s^2)^2$, $s\in \mathbb{R}$, where $\gamma_1$ and $\gamma_2$ are given positive constants.
\end{rmk}

Recently, there are some works devoted to the mathematical analysis of the Navier-Stokes-Cahn-Hilliard (NS-CH) system with the singular free energy density $\Phi$ satisfying Assumption \ref{assumption-1}, see \cite{AH1, AH3, AH5, AH4, FG1, MiZe2004} and the references cited therein.

In \cite{AH1}, H. Abels proved that there exist global weak solutions of the Navier-Stokes-Cahn-Hilliard system in 2D and 3D bounded domains for the singular free energy densities, and moreover, unique 'strong' solutions exist in 2D globally in time and in 3D locally in time. The existence of global weak solutions for inhomogeneous NS-CH system (where the density of the mixture depends enters the equation for singular chemical potential) was proved in 2D and 3D in \cite{AH3}.  And then, H. Abels, D. Depner, and H. Garcke \cite{AH5, AH4} investigated the existence of global weak solutions for a diffuse interface model for the flow of two viscous incompressible Newtonian fluids in a bounded domain in 2D and 3D, more generally, by allowing for a degenerate mobility. On the other hand, A. Miranville and S. Zelik \cite{MiZe2004} studied the long time behaviour of the Cahn-Hilliard equations with singular potentials (without convection), in which they were able, in two space dimensions, to separate the solutions from the singular values of the potential. S. Frigeri and M. Grasselli \cite{FG1} established the existence of a global weak solution of the nonlocal NS-CH system with no-slip and no-flux boundary conditions, and the existence of the global attractor for the 2D generalized semi-flow. Recently, A. Miranville and R. Temam \cite{MiTe2016} investigated the existence of weak solutions of the Cahn-Hilliard-Oono-Navier-Stokes euqations with singular nonlinear terms, and they pointed out that, it is not able to obtain the control of $\partial_tu$ In
$L^\infty(0,T; L^2(\Omega))$, and then the strict separation property doesn't hold (see Remark 3.2 in \cite{MiTe2016}). A. Giorgini, M. Grasselli and H. Wu \cite{GGW} researched the Cahn-Hilliard-Hele-Shaw system with singular potential, and they got the uniqueness and regularity of global weak solution by the so-called strict separation property in 2D.

While for the double-well free energy density case, there are numerous results on the study of the NS-CH system (see \cite{AMW, CG, FGR, G2016, GG, GGM, HH, LLG, LL} and their references). G. G. Gal and  M. Grasselli \cite{GG} considered NS-CH system with well-double potentials, and the asymptotic behavior of the solutions was studied in 2D bounded domain. C. Cao and C. G. Gal \cite{CG} investigated the NS-CH system without full viscosity and mobility which can be develop finite time singularities, and they proved the global existence and uniqueness of classical solution. X. Wang and Z. Zhang \cite{WZ} established existence and uniqueness of the 2D global (or 3D local) classical solution in periodic settings, as well as several blow-up criterions in the 3D case for the Hele-Shaw-Cahn-Hilliard (HS-CH) system, which can be formally viewed as an appropriate limit of the classical NS-CH system \cite{AMW, HH, LLG}. And then, the long-time behavior for the HS-CH system was demonstrated in \cite{WW}. Recently, C. G. Gal \cite{G2016} established the existence of globally defined weak solutions as well as well-posedness results for strong/classical solutions to the nonlocal incompressible Euler-Cahn-Hilliard equation (where the chemical potential $\mu=a_0\theta-J\ast \theta+\phi(\theta)$, $J$ is a spatial-dependent interaction kernel) in 2D bounded domains. The existence of a suitable
global energy solution to the NS-CH equations with moving contact lines was proved and the convergence of any such solution to a single
equilibrium was also established in \cite{GGM}.

In what follows, we will assume that $\kappa=m=1$ for simplicity, which follows that \eqref{1-NSCH-1} may be equivalently read as

\begin{equation}\label{NSCH-2}
  \begin{cases}
&\partial_t u+u\cdot\nabla u-\dive(2\nu(\theta)Du)+\nabla g=\mu\nabla\theta, \quad \forall\, (t, x)\in \mathbb{R}^{+}\times \mathbb{T}^2,
\\
&\partial_t\theta+ u\cdot\nabla\theta=\Delta\mu,\quad \mu=\phi(\theta)-\Delta\theta,\\
&\dive u=0,
\\
&(u, \theta)|_{t=0}=(u_0, \theta_0).
\end{cases}
\end{equation}

In this paper, we intend to establish the global well-posedness of the Navier-Stokes-Cahn-Hilliard system \eqref{NSCH-2} with the singular function $\phi$ in Assumption \ref{assumption-1}.
Our main result is stated as follows.

\begin{thm}\label{thm-glo-wp}
Under Assumption \ref{assumption-1}, let $s>1$, $(u_0,\theta_0)\in H^s(\mathbb{T}^2)\times H^{s}(\mathbb{T}^2)$, $\dive\, u_0=0$, $\int_{\mathbb{T}^2} \theta_0(x)\, dx=0$,
\begin{equation}\label{separate-condition-1}
\|\theta_0\|_{L^\infty(\mathbb{T}^2)}\leq 1-\delta_0
\end{equation}
for some $\delta_0\in(0,1)$, and $\nu(\cdot)$ is a smooth, positive function on $[-1, 1]$. Then the NS-CH system \eqref{NSCH-2} is globally well-posed. Moreover, there
holds that the solution $(u, \theta)$ to \eqref{NSCH-2} satisfies, for any $t>0$,
$\int_{\mathbb{T}^2} \theta(t, x)\, dx=0$, and
\begin{equation}\label{separate-condition-2}
\sup_{\tau \in [0, t]}\|\theta(\tau)\|_{L^\infty(\mathbb{T}^2)}\leq1-\delta,
\end{equation}
where $\delta=\delta(\delta_0,t)>0$,
and
\begin{align*}
&(u,\theta)\in (C(\mathbb{R}^+; H^s(\mathbb{T}^2))\cap L_{loc}^2(\mathbb{R}^+; H^{s+1}(\mathbb{T}^2)))
\\
&\qquad\qquad \times (C(\mathbb{R}^+; H^{s}(\mathbb{T}^2))\cap L_{loc}^2(\mathbb{R}^+; H^{s+2}(\mathbb{T}^2))).
\end{align*}
\end{thm}
\begin{rmk}\label{rmk-global-1}
It is worth pointing out that H. Abels in \cite{AH1} proved that unique 'strong' solutions globally exist in two dimensions when the initial data $\theta_0 \in H_N^2(\Omega)$, $\mathcal{E}(\theta_0)<+\infty$, $\mu(\theta_0) \in H^1(\Omega)$, and $u_0 \in V_2^2(\Omega)$, where there is no restriction on $\theta_0$ as in \eqref{separate-condition-1}, but also it loses the propagation of the high regularities of the solution because of lack of separating $\theta$ to the singular points of the potential $\phi$. In effect, the solution $(u, \, \theta)$ obtained in \cite{AH1} satisfies
$u \in  L^2(0, \infty; H^{2+s'}(\Omega))\cap H^1(0, \infty; H^{s'}(\Omega)) \cap BUC([0, \infty); H^{1+s-\varepsilon}(\Omega))$ for all $s' \in [0, \frac{1}{2})$ and all $\varepsilon>0$ as well as $\nabla^2\theta, \, \phi(\theta) \in L^{\infty}((0, \infty; L^{r}(\Omega)) )$ for every $1 < r < \infty$.
\end{rmk}
Let's explain the main idea of the proof of Theorem \ref{thm-glo-wp}. Under the assumptions in Theorem \ref{thm-glo-wp}, basic energy estimates yield that the solution $(u, \theta)$ satisfies $\|\theta\|_{L^\infty((0, T)\times \mathbb{T}^2)}\leq  1$, and $u, \, \theta\in L^\infty((0,T); H^1(\mathbb{T}^2))\cap L^2((0,T); H^2(\mathbb{T}^2))$, $u_t, \, \nabla\mu \in L^2((0,T); L^2(\mathbb{T}^2))$. While, from the local well-posedness theory (see Theorem \ref{thm-lp-nsch}), uniqueness and regularity depend strongly on the concentration $\theta$ away from the singular values $\pm 1$ of the function  $\phi$ in the Cahn-Hilliard equation with convection term. Hence, in order to extend the local solution to the global one, we need to prove that not only $(\nabla u, \Delta\theta)$ is bounded in $L^2((0, T); L^{\infty}(\mathbb{T}^2))$, but also $\theta$ is separate from the singular points of $\phi$ for any existence time. Due to lack of the maximum principle for the $\theta$-equation in \eqref{NSCH-2} (as a fourth-order parabolic equation), we construct its approximate second-order parabolic equation, and use comparison principle and the basic energy estimates $u\in L^\infty((0,T); H^1(\mathbb{T}^2))\cap L^2((0,T); H^2(\mathbb{T}^2))$ and $u_t \in L^2((0,T); L^2(\mathbb{T}^2))$ to separate $\theta$ from the singular values of the function  $\phi$, where the Orlicz embedding theorem (Lemma \ref{thm-O-1}) plays a crucial role. From this, the singular problem is reduced to a regular problem. With this in hand, we may readily prove $\theta \in L^2((0,T); H^{2+s_0}(\mathbb{T}^2))$ for any $s_0 \in (1, \frac{3}{2}]$ by using the energy estimate, which bounds $\|\Delta \theta\|_{L^2((0, T); L^{\infty}(\mathbb{T}^2))}$. Another difficulty stems from the Lipschitz estimate of the velocity, since $H^2(\mathbb{T}^2)  \looparrowright Lip(\mathbb{T}^2)$. Fortunately, this can be solved by combining energy estimates with the Logarithmic Sobolev interpolation inequality (Lemma \ref{lem-logarithmic}).

The rest of the paper is organized as follows. We review in Section \ref{sect-2} some preliminary results such as basic calculus in Sobolev spaces, Orlicz embedding theorem, basic properties of the bi-harmonic heat flow, and some properties of the singular free energy density $\phi$. In
Section \ref{sect-3} we present basic energy estimates of the Navier-Stokes-Cahn-Hilliard system \eqref{NSCH-2}. Section \ref{sect-4} is devoted to the local in time well-posedness of the NS-CH system, which proof requires some refined estimates relying on Littlewood-Paley analysis in Appendix A. Based on energy estimates in Section \ref{sect-3}, we obtain uniform $L^\infty$-bounds of concentrations away from singular points of the function $\phi$ in Section \ref{sect-5}. Finally, the global well-posedness of the system \eqref{NSCH-2} is proved in Section \ref{sect-6}.

{\bf Notations:}
Let $A$, $B$ be two operators, we denote $[A, B] =AB-BA$ the commutator between $A$ and $B$. For $a  \lesssim b$, we mean that there is a uniform constant $C$, which may be different on different lines, such that $a \leq C\,b$. We shall denote by $(a, b)$(or $(a, b)_{L^2}$) the $L^2(\mathbb{T}^2)$ inner product of $a$ and $b$. For $X$ a Banach space and $I$ an interval of $\mathbb{R}$, we denote by
$C(I;\,X)$ the set of continuous functions on $I$ with
values in $X,$ and by $C_b(I;\,X)$ the subset of bounded
functions of $C(I;\,X).$ For $q\in[1,+\infty],$ the
notation $L^q(I;\,X)$ stands for the set of measurable functions on
$I$ with values in $X,$ such that $t\longmapsto\|f(t)\|_{X}$ belongs
to $L^q(I).$ For a vector $v =(v_1, v_2) \in X$, we mean that all the components $v_i$ ($i=1, 2$) of $v$ belong to the space $X$.
Moreover, $m(f)=\frac{1}{|\mathbb{T}^2|}\int_{\mathbb{T}^2}f(x)dx$ is the mean value of $f$ on $\mathbb{T}^2$. We always denote the Fourier transform of a function $u$ by $\widehat{u}$ or $\mathcal{F}(u)$.

\renewcommand{\theequation}{\thesection.\arabic{equation}}
\setcounter{equation}{0}
\section{Preliminaries}\label{sect-2}

In this section, we recall some preliminary results that are useful throughout paper.

\subsection{Some calculus in Sobolev spaces}

\begin{lem}[Moser-type estimates, \cite{KM}]\label{lem-3}
Let $s>0$. Then the following two estimates are true:

(i) $\|uv\|_{H^s(\mathbb{T}^d)}\leq C (\|u\|_{L^\infty(\mathbb{T}^d)}\|v\|_{H^s(\mathbb{T}^d)}+\|u\|_{H^s(\mathbb{T}^d)}\|v\|_{L^\infty(\mathbb{T}^d)})$;

(ii) $\|uv\|_{H^s(\mathbb{T}^d)}\leq C\|u\|_{H^s(\mathbb{T}^d)}\|v\|_{H^s(\mathbb{T}^d)} \quad \mbox{for all} \quad s>\frac{d}{2}$;

\noindent where all the constants $C$s are independent of $u$ and $v$.
\end{lem}

\begin{lem}[Commutator estimate, \cite{Ka-Pon-88}]\label{lemma-comm}
Let $\Lambda^s:=(1-\Delta)^{\frac{s}{2}}$ with $s>0$. Then the  following estimate holds:
\begin{equation*}
\|[\Lambda^s, u]v\|_{L^2(\mathbb{T}^d)} \leq C(\|u\|_{H^s(\mathbb{T}^d)}\|v\|_{L^\infty(\mathbb{T}^d)}+\|\grad u\|_{L^\infty(\mathbb{T}^d)}\|v\|_{H^{s-1}(\mathbb{T}^d)}),
\end{equation*}
where the constant $C$ is independent of $u$ and $v$.

\end{lem}

\begin{lem}[Logarithmic Sobolev interpolation inequality, \cite{Brezis1980}]\label{lem-logarithmic}
For any $f\in {H^s}(\mathbb{T}^d)$
with $s> \frac{d}{2}$, there holds
\begin{equation}\label{Log-1}
\|f\|_{L^\infty}\leq C(1+\|f\|_{H^{\frac{d}{2}}})\log^{\frac{1}{2}}(e+\|f\|_{H^s})).
\end{equation}
\end{lem}

The action of smooth functions on the space $H^s$ may be stated as follows.

\begin{lem}[\cite{BCD, DR}]\label{lem-4}
Let $I$ be an open interval of $\mathbb{R}$ and $F$ : $I\rightarrow\mathbb{R}$. Let $s>0$ and $\sigma>$ be the smallest integer such that $\sigma>s$.
Assume that $F''$ belongs to $W^{\sigma,\infty}(I;\mathbb{R})$. Let $u,\,v\in H^s(\mathbb{T}^d)\cap L^\infty(\mathbb{T}^d)$ have values in $J\subset I$. There exists a constant $C=C(s,I,J,N)$ such that
\begin{equation}
\|F(u)\|_{H^s}\leq C(1+\|u\|_{L^\infty})^\sigma\|F''\|_{W^{\sigma,\infty}(I)}\|u\|_{H^s},\quad \hbox{if}\quad F(0)=0,\label{2-2}
\end{equation}
and
\begin{align}
&\|F\circ v-F\circ u\|_{H^s}\leq C(1+\|u\|_{L^\infty}+\|v\|_{L^\infty})^\sigma\|F''\|_{W^{\sigma,\infty}(I)}\nonumber
\\
&\qquad\times(\|u-v\|_{H^s}\sup_{\tau\in[0,1]}\|v+\tau(u-v)\|_{L^\infty}+\|u-v\|_{L^\infty}\sup_{\tau\in[0,1]}\|v+\tau(u-v)\|_{H^s}).\label{2-3}
\end{align}
\end{lem}

\subsection{Orlicz embedding theorem}
In this subsection, we recall the definitions of Orlicz spaces and classes (see \cite{TNS}). Let $\phi(t)$ be a real-valued continuous, convex, even function of the real variable $t$, satisfying
\begin{equation*}\label{2-40}
\lim_{t\rightarrow0}\frac{\phi(t)}{t}=0,\quad\lim_{t\rightarrow\infty}\frac{\phi(t)}{t}=\infty.
\end{equation*}
Then the Orlicz class $L_{\phi}(\Omega)$ is defined as follows
\begin{equation*}\label{2-40-1}
L_{\phi}(\Omega):=\bigg\{u(x) \, \mbox{is measurable in}\, \Omega|\int_\Omega\phi(u(x))dx<\infty\bigg\}.
\end{equation*}
The Orlicz space $L_{\phi^*}(\Omega)$ may be defined as the linear hull of  $L_{\phi}(\Omega)$ together with the Luxembourg norm
\begin{equation*}\label{2-41}
\|u\|_{ L_{\phi^*}(\Omega)}:=\inf\{k>0;\,\int_\Omega \phi(k^{-1} u(k))\,dx\leq 1\}.
\end{equation*}
 $L_{\phi^*}(\Omega)$ is a Banach space under \eqref{2-41}. We call $\phi(t)$ a defining function for  $L_{\phi^*}(\Omega)$.

 If for any two defining functions $\phi(t)$, $\psi(s)$, we have for every $\lambda>0$
 \begin{equation*}
 \lim_{t\rightarrow\infty}\frac{\phi(\lambda t)}{\psi(t)}=\infty,
 \end{equation*}
then we write $\psi\prec\phi$. Note that this means that
$L_{\phi^*}(\Omega)\varsubsetneqq  L_{\psi^*}(\Omega)$.

If a sequence $\{u_n(x)\}_{n \in \mathbb{N}}\subset L_{\phi^*}(\Omega)$ converges in measure and is bounded in $ L_{\phi^*}(\Omega)$,
then $u_n(x)$ converges in  $L_{\psi^*}(\Omega)$ for any $\psi\prec\phi$.

A sequence of functions $\{u_n(x)\}_{n \in \mathbb{N}}\subset L_{\phi}(\Omega)$ is said to be mean convergent to $u(x)\in L_{\phi^*}(\Omega)$ if
\begin{equation*}
\int_\Omega\phi(u_n-u)dx\rightarrow0\quad\hbox{as}\quad n\rightarrow\infty.\label{2-42}
\end{equation*}
Mean convergence is a weaker property than norm convergence although for a large class of Orlicz spaces which includes the $L_p$ spaces, $p>1$,
the two notions are equivalent.

\begin{prop}[Young's inequality, \cite{AA}]\label{prop-O-Y}
\begin{equation}\label{2-4}
p\cdot q\leq A(p)+\tilde{A}(q), \quad\forall~p,q\geq0,
\end{equation}
where
\begin{equation}
A(s):=e^s-s-1, \quad\tilde{A}(s):=(1+s)\ln(1+s)-s.\label{2-5}
\end{equation}
\end{prop}
Note that $\forall\, s\geq0$
\begin{align}
\tilde{A}(s)=&(1+s)\ln(1+s)-s=s\ln(1+s)+(\ln(1+s)-s)\leq s\ln(1+s).\label{2-6}
\end{align}

\begin{lem}[Orlicz embedding theorem, \cite{TNS}]
\label{thm-O-Em}
Let $\Omega$ satisfy a cone condition. Then the Sobolev space $W^{k, p}(\Omega)$, where $n=kp$, $k\geq 0$ is a positive integer, $p \in (1, +\infty)$, may be continuously imbedded in
Orlicz space $L_{\phi*}(\Omega)$ where
\begin{equation*}
\phi(t)=e^{|t|^{n/(n-1)}}-1.
\end{equation*}
Furthermore $W^{k, p}(\Omega)$ may be continuously imbedded in the sense of mean convergence in any Orlicz class $L_{\psi}(\Omega)$ where $\psi(t)\leq\phi(\lambda t)$ for some  $\lambda>0$. The imbedding
into $L_{\psi*}(\Omega)$ for any $\psi\prec\phi$ is compact.
\end{lem}

\begin{rmk}
\label{rmk-1}
In this paper, we will use the Orlicz embedding Theorem in 2-D periodic domain from Theorem \ref{thm-O-Em}, which is contained in the following version of the   Orlicz embedding theorem.
\end{rmk}
\begin{lem}[Orlicz embedding theorem, \cite{TNS}]
\label{thm-O-1}
Let $\Omega\subset\mathbb{R}^2$ is a bounded domain satisfying a cone condition, $\beta>0$. Then
\begin{equation*}\label{2-7}
\int_\Omega e^{\beta|v(x)|}\,dx\leq C_{\beta, \Omega} e^{C_{\beta, \Omega}\|v\|^2_{H^1(\Omega)}}, \quad\forall~v\in H^1(\Omega),
\end{equation*}
where $C_{\beta, \Omega}$ depends only on $\beta$ and $|\Omega|$.
\end{lem}

\subsection{Basic properties of the bi-harmonic heat flow}

Let's now recall some fundamental properties of the bi-harmonic heat flow on periodic domains.

Consider the solution $\theta(t, x)$ to the bi-harmonic heat equation:
\begin{equation*}\label{bi-harm-fund}
\begin{cases}
 & (\partial_t +\Delta^2) \theta=0,\quad \forall \, (t, x) \in \mathbb{R}^+\times\mathbb{T}^2,\\
 &\theta|_{t=0}=\theta_0,
\end{cases}
\end{equation*}
where initial data $\theta_0 \in H^s(\mathbb{T}^2)$ with $s > 1$ and $\int_{\mathbb{T}^2} \theta_0(x)\, dx=0$.
Then we have
\begin{equation*}\label{bi-harm-fund-1}
\begin{split}
&\widehat\theta(t, n)=e^{-t|n|^4}\widehat\theta_0(n), \quad\forall \,\, n=(n_1, n_2) \in \mathbb{Z}^2,
\end{split}
\end{equation*}
where $\widehat{f}(n):=(2\pi)^{-2}\int_{\mathbb{T}^2} f(x) e^{-i\,n\cdot x}\, dx$ for $\forall\, f \in L^1(\mathbb{T}^2)$, which implies that
\begin{equation}\label{bi-harm-fund-2}
\begin{split}
&\theta(t, x)=e^{-t\Delta^2} \theta_0(x)=\sum_{n \in \mathbb{Z}^2}e^{-t|n|^4}\widehat\theta_0(n)e^{i\, n\cdot x}, \quad\forall \,\, x \in \mathbb{T}^2,
\end{split}
\end{equation}
and $\int_{\mathbb{T}^2} \theta(t, x)\, dx=0$ for any $t>0$.

Moreover, we claim that
\begin{equation}\label{bi-harm-fund-3}
\|\theta(t, \cdot)-\theta_0(\cdot)\|_{L^{\infty}(\mathbb{T}^2)} \rightarrow 0\quad (\mbox{as} \quad t\rightarrow 0^+).
\end{equation}
In effect, since $\theta_0 \in H^s(\mathbb{T}^2)$ with $s>1$, we find $\theta_0(x)=\sum_{n \in \mathbb{Z}^2}\widehat\theta_0(n)e^{i\, n\cdot x}$ for any $x \in \mathbb{T}^2$, which follows from \eqref{bi-harm-fund-2} that
\begin{equation*}\label{bi-harm-fund-4}
\begin{split}
&\theta(t, x)-\theta_0(x)=\sum_{n \in \mathbb{Z}^2}(e^{-t|n|^4}-1)\widehat\theta_0(n)e^{i\, n\cdot x}, \quad\forall \,\, x \in \mathbb{T}^2.
\end{split}
\end{equation*}
From this, we find that for any $t>0$, $x \in \mathbb{T}^2$
\begin{equation}\label{bi-harm-fund-5}
\begin{split}
&|\theta(t, x)-\theta_0(x)|\leq \sum_{n \in \mathbb{Z}^2}|e^{-t|n|^4}-1||\widehat\theta_0(n)|.
\end{split}
\end{equation}
Since
\begin{equation*}\label{bi-harm-fund-6}
\begin{split}
\sum_{n \in \mathbb{Z}^2}|e^{-t|n|^4}-1||\widehat\theta_0(n)|\leq \sum_{n \in \mathbb{Z}^2, \, |n| \geq 1}|\widehat\theta_0(n)| \leq C_s\|\theta_0\|_{H^s},
\end{split}
\end{equation*}
where we have used the fact that $s>1$, Lebesgue's dominated convergence theorem ensures that
\begin{equation*}\label{bi-harm-fund-7}
\begin{split}
\lim_{t\rightarrow 0^+}\sum_{n \in \mathbb{Z}^2}|e^{-t|n|^4}-1||\widehat\theta_0(n)|=0,
\end{split}
\end{equation*}
which along with \eqref{bi-harm-fund-5} yields \eqref{bi-harm-fund-3}.
\begin{rmk}\label{rmk-bound-sep}
According to \eqref{bi-harm-fund-3}, we know that, if initial data $\theta_0 \in H^s(\mathbb{T}^2)$ with $s > 1$ and $\int_{\mathbb{T}^2} \theta_0(x)\, dx=0$, then for any given $\delta_0 \in (0, 1)$, there is a positive time $T_1$ such that, for any $t\in [0, T_1]$, there holds that
  \begin{equation}\label{bi-harm-fund-10}
\|e^{-t\Delta^2}\theta_0\|_{L^{\infty}(\mathbb{T}^2)} \leq \|\theta_0\|_{L^{\infty}(\mathbb{T}^2)}+\frac{1}{4}\delta_0.
\end{equation}
\end{rmk}

\subsection{Some properties of the singular free energy density $\phi$}

In this subsection, we record some elementary properties in terms of the singular function $\phi$, which also appears in \cite{MiZe2004}. We will give them rigorous proofs for the sake of completeness.

\begin{prop}\label{prop-2}
Given the function $\phi=\Phi'$ with \eqref{1-9}, $\theta \in L^\infty(\mathbb{T}^2)$ with $\|\theta\|_{L^\infty(\mathbb{T}^2)}\leq 1$,
$\phi(\theta)\in L^1(\mathbb{T}^2)$, and $f\in L^1(\mathbb{T}^2)$ satisfy
\begin{equation}\label{2-8}
m(\theta)=0, \quad \phi(\theta)-m(\phi(\theta))=f \quad (\forall\, x\in\mathbb{T}^2).
\end{equation}
Then there holds
\begin{equation}\label{2-9}
|m(\phi(\theta))|\leq C(\|f\|_{L^1}+1)
\end{equation}
for some positive constant $C$ depending only on $\alpha$ in \eqref{1-9}.
\end{prop}

\begin{proof}
According to \eqref{1-9}, there exist two constants $c_1, \, c_2 \in (0, 1)$  such that $\phi(\theta)>1$ for $\theta \in [c_1, 1)$ and  $\phi(\theta)<-1$ for $\theta \in (-1, -c_2]$, and
\begin{equation}\label{2-9-0}
\begin{split}
(\phi(s)-\phi(0)+\alpha s)' \geq 0 \quad \forall \, s \in (-1, 1),
\end{split}
\end{equation}
 which follows
\begin{equation}\label{2-9-1}
\begin{split}
&(\phi(\theta)-\phi(0)+\alpha \theta)\theta \geq 0 \quad \forall \, \theta \in (-1, 1).
\end{split}
\end{equation}
Since $\phi(\cdot)$ is continuous and then bounded on $[-c_2, c_1]$, we find
\begin{equation}\label{2-9-2}
|\int_{\theta \in [-c_2, c_1]}\phi(\theta)\, dx|\leq C.
\end{equation}
For $\theta\in [c_1, 1)$, we have $2c_1^{-1}\phi(\theta)\theta-\phi(\theta)=(2c_1^{-1}\theta-1)\phi(\theta)\geq0$, which yields $1 \leq \phi(\theta) \leq 2c_1^{-1}\phi(\theta)\theta$, and then ensures from \eqref{2-9-1}
\begin{equation*}\label{2-9-3}
\begin{split}
&|\int_{\theta\in [c_1, 1)} \phi(\theta)\, dx| \leq 2c_1^{-1} \int_{\theta\in (c_1, 1)}\phi(\theta)\theta\, dx \\
&=2c_1^{-1} \int_{\theta\in (c_1, 1)}(\phi(\theta)-\phi(0)+\alpha \theta)\theta\, dx +2c_1^{-1} \int_{\theta\in (c_1, 1)}(\phi(0)-\alpha \theta)\theta\, dx\\
&\leq 2c_1^{-1} \int_{\mathbb{T}^2}(\phi(\theta)-\phi(0)+\alpha\, \theta)\theta\, dx+2c_1^{-1} \int_{\mathbb{T}^2}|\phi(0)||\theta|\, dx.
\end{split}
\end{equation*}
Thus, it follows that
\begin{equation*}\label{2-9-4}
\begin{split}
&|\int_{\theta\in [c_1, 1)} \phi(\theta)\, dx| \leq 2c_1^{-1} \int_{\mathbb{T}^2}\phi(\theta)\,\theta\, dx+4c_1^{-1} |\phi(0)|\int_{\mathbb{T}^2}|\theta|\, dx+ 2c_1^{-1} \alpha\, \int_{\mathbb{T}^2}|\theta|^2\, dx,
\end{split}
\end{equation*}
which along with $\|\theta\|_{L^\infty(\mathbb{T}^2)}\leq 1$ leads to
\begin{equation}\label{2-9-5}
\begin{split}
&|\int_{\theta\in [c_1, 1)} \phi(\theta)\, dx| \leq 2C( \int_{\mathbb{T}^2}\phi(\theta)\,\theta\, dx+1).
\end{split}
\end{equation}
Similarly we may get
\begin{equation}\label{2-9-6}
\begin{split}
&|\int_{\theta\in (-1, -c_2]} \phi(\theta)\, dx| \leq 2C( \int_{\mathbb{T}^2}\phi(\theta)\,\theta\, dx+1).
\end{split}
\end{equation}
Combining \eqref{2-9-6} with \eqref{2-9-5} and \eqref{2-9-2} ensures
\begin{equation}\label{2-10}
\begin{split}
|m(\phi(\theta))| =\frac{1}{|\mathbb{T}|^2}|\int_{\mathbb{T}^2} \phi(\theta)\, dx| \leq C( \int_{\mathbb{T}^2}\phi(\theta)\theta\, dx +1)
\end{split}
\end{equation}
Next, we will bound the integral on the right hand side of \eqref{2-10}. In effect, multiplying \eqref{2-8} by $\theta$ and integrating the resulting equation over $\mathbb{T}^2$ yield
\begin{equation*}\label{2-11}
\int_{\mathbb{T}^2}\phi(\theta)\,\theta\, dx= m(\phi)\int_{\mathbb{T}^2}\theta\,dx+\int_{\mathbb{T}^2}f\, \theta\,dx=\int_{\mathbb{T}^2}f\, \theta\, dx,
\end{equation*}
where we have used $m(\theta)=0$, from which, it follows
\begin{equation}\label{2-11-1}
\begin{split}
&|\int_{\mathbb{T}^2} \phi(\theta)\theta\, dx| \leq \|\theta\|_{L^{\infty}}\|f\|_{L^1} \leq \|f\|_{L^1}.
\end{split}
\end{equation}
Combining \eqref{2-11-1} with \eqref{2-10} gives rise to \eqref{2-9}, which completes the proof of Proposition \ref{prop-2}.
\end{proof}

In order to make the solution $\theta$ separate away from the singular points, we want to apply the comparison principle to the convective Cahn-Hilliard equation (see Section 3), which is based on the following comparison principle of the first ordinary differential equation:
\begin{equation}\label{2-12}
\varepsilon y'+\phi(y)=h, \quad y(0)=y_0,\quad |y_0|<1.
\end{equation}
\begin{prop}[Comparison principle]\label{prop-3}
Let the function $\phi$ satisfy
\begin{equation}\label{2-12-0}
\lim_{y\rightarrow \pm 1^{\mp}} \phi(y)=\pm \infty,
\end{equation}
 and let us assume that $\varepsilon > 0$ and
\begin{equation}\label{2-13}
|y_0|\leq 1-\delta_0\quad \hbox{and}\quad h\in L^\infty([0,T])
\end{equation}
with some positive constant $\delta_0$ and $T$. Then for any solution $y(t)$ to \eqref{2-12} with $|y(t) | \leq 1$, there exists a constant $\delta=\delta(\delta_0, \,\|h\|_{L^\infty([0,T])})>0$ independent of $\varepsilon$,
such that
\begin{equation}
|y(t)|\leq 1-\delta,\quad\forall t\in[0,T].\label{2-14}
\end{equation}
\end{prop}
\begin{proof}
We rewrite \eqref{2-12} as
\begin{equation}\label{2-12-1}
\begin{cases}
&\varepsilon y'=h-\phi(y), \\
&y|_{t=0}=y_0
\end{cases}
\end{equation}
with $|y_0|<1-\delta_0$.

Consider first the case $y \geq 0$. If $h-\phi(y)>0$, then $\phi(y) \leq \|h\|_{L^{\infty}([0, T])}$, which follows from \eqref{2-12-0} that, there exists a constant $\delta=\delta(\delta_0, \,\|h\|_{L^\infty([0,T])})>0$ which is independent of $\varepsilon$,
such that
\begin{equation*}
y(t)\leq 1-\delta,\quad\forall t\in[0,T].
\end{equation*}
On the other hand,  if $h-\phi(y)\leq 0$, applying the comparison principle of the first ordinary differential equation between \eqref{2-12-1} and the following equation
\begin{equation}\label{2-12-2}
  \begin{cases}
    &\varepsilon y'=0 (\geq h-\phi(y)),\\
    &y|_{t=0}=y_0,
  \end{cases}
\end{equation}
which has a unique solution $y\equiv y_0$, we get $0 \leq y \leq y_0\leq 1-\delta_0.$

While for the case $y < 0$: if $h-\phi(y)< 0$, then $\phi(y) \geq - \|h\|_{L^{\infty}([0, T])}$, which follows from \eqref{2-12-0} that, there exists a constant $\delta=\delta(\delta_0$, $\|h\|_{L^\infty([0,T])})>0$ which is independent of $\varepsilon$,
such that
\begin{equation*}
0 \geq y(t)\geq -(1-\delta),\quad\forall t\in[0,T].
\end{equation*}
Otherwise,  if $h-\phi(y)\geq 0$, applying the comparison principle of the first ordinary differential equation between \eqref{2-12-1} and \eqref{2-12-2}, we get $0 \geq y \geq y_0\geq -(1-\delta_0).$
Therefore, we complete the proof of Proposition \ref{prop-3}.
\end{proof}

\begin{prop}\label{prop-4}
Let the function $\phi=\Phi'$ satisfy \eqref{1-9} and $h$ belong to $L^2([0,T])$. Then, for any solution $y(t)$ of \eqref{2-12} with $|y| \leq 1$, there holds
\begin{equation}\label{2-15}
\int^T_0|\phi(y)|^2dt\leq C_T (1+\|h\|^2_{L^2([0,T])}),
\end{equation}
where the constant $C_T$ is independent of $\varepsilon$.
\end{prop}

\begin{proof}
Multiplying \eqref{2-12} by $\phi(y)$ and integrating over $[0,T]$ yield
\begin{equation}\label{2-16}
\varepsilon\Phi(y(T))+\int^T_0|\phi(y)|^2=\varepsilon\Phi(y(0))+\int^T_0 h(t)\phi(y(t))dt.
\end{equation}
Letting $\tilde{\Phi}(y):=\int_0^y(\phi(s)-\phi(0)+\alpha s)\, ds=\Phi(y)-\phi(0)\, y +\frac{\alpha}{2}|y|^2$ for any $y \in (-1, 1)$, we infer from \eqref{2-9-0} that
\begin{equation}\label{2-16-1}
\tilde{\Phi}(y) \geq 0 \quad \forall \, y \in (-1, 1),
\end{equation}
it follows from \eqref{2-16} that
\begin{equation*}\label{2-17}
\begin{split}
&\varepsilon\tilde\Phi(y(T))+\int^T_0|\phi(y)|^2=\varepsilon\Phi(y(0))+\varepsilon(-\phi(0)\, y(T) +\alpha|y(T)|^2)+\int^T_0 h(t)\phi(y(t))dt\\
&\leq \varepsilon\Phi(y(0))+\frac{1}{2}\int^T_0 |h(t)|^2dt+\frac{1}{2}\int^T_0 |\phi(y)|^2dt+\varepsilon(|\phi(0)|\, |y(T)| +\alpha|y(T)|^2).
\end{split}
\end{equation*}
Hence, we get from the facts that $|y|\leq1$ and $\tilde{\Phi}(y) \geq0$ that
\begin{equation*}\label{2-18}
\int^T_0|\phi(y)|^2dt\leq C_{T}(1+\|h\|^2_{L^2([0,T])}).
\end{equation*}
This ends the proof of Proposition \ref{prop-4}.
\end{proof}

\renewcommand{\theequation}{\thesection.\arabic{equation}}
\setcounter{equation}{0}
\section{Basic energy estimates}\label{sect-3}

\begin{lem}\label{lem-unif-u-1}
Under the assumptions in Theorem \ref{thm-glo-wp}, let $(u,\theta)$ be a smooth solution to the Navier-Stokes-Cahn-Hilliard
system \eqref{NSCH-2} on $[0,T)$ for $0 <T <+\infty$ satisfying
\begin{equation*}
\begin{split}
&(u, \theta)\in  \bigg(C([0,T); H^s(\mathbb{T}^2))\cap L^2_{loc}([0,T); H^{s+1}(\mathbb{T}^2)) \bigg) \\
&\qquad\qquad\qquad\qquad\qquad\times
 \bigg(C([0,T); H^s(\mathbb{T}^2)) \cap L^2_{loc}([0,T); H^{s+2}(\mathbb{T}^2))\bigg).
 \end{split}
\end{equation*}
Then there holds
\begin{equation}\label{3-56-2}
\|\theta\|_{L^\infty((0, T)\times \mathbb{T}^2)}\leq  1,
\end{equation}
and moreover,
\begin{equation}\label{3-56-3}
\begin{split}
&\|u\|^2_{L^\infty((0,T); H^1(\mathbb{T}^2))}+\|u_t\|^2_{L^2((0,T); L^2(\mathbb{T}^2))}+\|u\|^2_{L^2((0,T); H^2(\mathbb{T}^2))}
\\
&+\|\theta\|^2_{L^\infty((0,T); H^1(\mathbb{T}^2))}+\|\theta\|^2_{L^2((0,T); H^{2}(\mathbb{T}^2))}+\|\nabla\mu\|^2_{L^2((0,T); L^2(\mathbb{T}^2))}\leq C_T(u_0,\theta_0).
\end{split}
\end{equation}
\end{lem}
\begin{proof}
Multiplying the $\theta$-equation of \eqref{NSCH-2} by $\theta$ and integrating the resulting equation on $\mathbb{T}^2$ yield
\begin{equation} \label{3-57}
\frac{1}{2}\frac{d}{dt}\|\theta\|^2_{L^2}+\int_{\mathbb{T}^2} u \cdot\nabla \theta\, \theta\, dx=\int_{\mathbb{T}^2}\theta \,\Delta\mu \,dx.
\end{equation}
Thanks to $\nabla \cdot u=0$ and $\phi'\geq-\alpha$, we find that
\begin{equation}\label{3-58}
\int_{\mathbb{T}^2} u \cdot\nabla \theta\, \theta\, dx=0,
\end{equation}
and
\begin{equation}\label{3-59}
\begin{split}
&\int_{\mathbb{T}^2} \theta \,\Delta\mu \,dx=-\int_{\mathbb{T}^2} |\Delta\theta|^2\, dx-\int_{\mathbb{T}^2}  \phi'(\theta)|\nabla\theta|^2\,dx\leq-\|\Delta\theta\|^2_{L^2}+\alpha\|\nabla\theta\|^2_{L^2}\\
&\leq-\frac{3}{4}\|\Delta\theta\|^2_{L^2}+C\|\theta\|^2_{L^2},
\end{split}
\end{equation}
where we have used the interpolation inequality $\|\nabla\theta\|_{L^2(\mathbb{T}^2)}^2 \leq C \|\theta\|_{L^2(\mathbb{T}^2)}\|\Delta\theta\|_{L^2(\mathbb{T}^2)}$ and Young's inequality in the last inequality.

Thus, it follows  that
\begin{equation}\label{3-60}
\frac{d}{dt}\|\theta\|^2_{L^2}+\|\theta\|^2_{H^2} \leq C\|\theta\|^2_{L^2},
\end{equation}
which along with Gronwall's inequality gives rise to
\begin{equation}\label{3-61}
\sup_{\tau\in [0, t]}\|\theta(\tau)\|^2_{L^2}+\int^t_0\|\theta\|^2_{H^2}\,d\tau \leq C e^{Ct}\|\theta_0\|_{L^2}^2.
\end{equation}
On the other hand, multiplying the $\theta$-equation and the $u$-equation of \eqref{NSCH-2} by $\mu$ and $u$ respectively, and then integrating the resulting equations over $\mathbb{T}^2$, we get
\begin{equation*}
\begin{split}
&\int_{\mathbb{T}^2} \partial_t\theta(-\Delta\theta+\phi(\theta))\,dx+\int_{\mathbb{T}^2}  (u\cdot\nabla)\theta\, \mu\, dx=\int_{\mathbb{T}^2}   |\nabla\mu|^2\, dx
\end{split}
\end{equation*}
and
\begin{equation*}
\begin{split}
&\frac{1}{2}\frac{d}{dt}\|u\|^2_{L^2}+\int_{\mathbb{T}^2} (u\cdot\nabla u)\cdot u\,dx-\int_{\mathbb{T}^2} \dive(2\nu(\theta)D(u))\cdot u\,dx=\int_{\mathbb{T}^2} \mu\nabla\theta\cdot u\,dx,
\end{split}
\end{equation*}
which yields
\begin{equation}\label{3-67}
\frac{d}{dt}\bigg(\frac{1}{2}\|\nabla\theta\|^2_{L^2}+\frac{1}{2}\|u\|^2_{L^2}+\int_{\mathbb{T}^2} \Phi(\theta)\,dx\bigg)
+\int_{\mathbb{T}^2} ( |\nabla\mu|^2+2\nu(\theta)|D(u)|^2)\,dx=0.
\end{equation}
Hence, we establish from $\|\nabla u\|^2_{L^2}\leq C\|Du\|^2_{L^2}$ that
\begin{equation}\label{3-68}
\begin{split}
&\frac{1}{2}(\|\nabla\theta(t)\|^2_{L^2}+\|u(t)\|^2_{L^2})+\int_{\mathbb{T}^2} \Phi(\theta(t))\,dx+\int^t_0(\|\nabla\mu(\tau)\|^2_{L^2}+\nu_0\|\nabla u(\tau)\|^2_{L^2})\,d\tau
\\
&\leq \frac{1}{2}(\|\nabla\theta_0\|^2_{L^2}+\|u_0\|^2_{L^2}+2\int_{\mathbb{T}^2} \Phi(\theta_0)\,dx)
\end{split}
\end{equation}
for some positive constant $\nu_0$.
Letting $\tilde{\Phi}(\theta)=\Phi(\theta)-\phi(0)\theta+\frac{\alpha}{2}\theta^2$, where $\tilde{\Phi}(\theta)\geq 0$ from \eqref{2-16-1}, we infer from \eqref{3-68} that
\begin{equation*}
\begin{split}
&\frac{1}{2}(\|\nabla\theta(t)\|^2_{L^2}+\|u(t)\|^2_{L^2})+\int_{\mathbb{T}^2} \tilde{\Phi}(\theta(t))\,dx+\int^t_0\|\nabla\mu(s)\|^2_{L^2}\,ds+\nu_0\int^t_0\|\nabla u(s)\|^2_{L^2}\,ds
\\
&\leq \frac{1}{2}(\|\nabla\theta_0\|^2_{L^2}+\|u_0\|^2_{L^2})+\int_{\mathbb{T}^2} \Phi(\theta_0)\,dx+C\|\theta(t)\|^2_{L^2},
\end{split}
\end{equation*}
which along with \eqref{3-61} implies
\begin{equation}\label{3-69}
\begin{split}
&\frac{1}{2}(\|\nabla\theta(t)\|^2_{L^2}+\|u(t)\|^2_{L^2})+\int_{\mathbb{T}^2} \tilde{\Phi}(\theta(t))\,dx+\int^t_0\|\nabla\mu(s)\|^2_{L^2}\,ds+\nu_0\int^t_0\|\nabla u(s)\|^2_{L^2}\,ds
\\
&\leq \frac{1}{2}(\|\nabla\theta_0\|^2_{L^2}+\|u_0\|^2_{L^2})+\int_{\mathbb{T}^2} \Phi(\theta_0)\,dx+ C e^{Ct}\|\theta_0\|_{L^2}^2 \leq C_T,
\end{split}
\end{equation}
and then \eqref{3-56-2} holds from \eqref{1-9}, the definition of $\Phi$.

Let us now estimate $\|\nabla u\|^2_{L^{\infty}([0,T]; L^2(\mathbb{T}^2))}$.
Multiplying the $u$-equation of \eqref{NSCH-2} by $-\Delta u$ and then integrating the resulting equation over $\mathbb{T}^2$ provide that
\begin{equation}\label{3-70}
\begin{split}
&\frac{1}{2}\frac{d}{dt}\|\nabla u\|^2_{L^2}-\int_{\mathbb{T}^2}  (u\cdot\nabla) u\cdot\Delta u\,dx+\int_{\mathbb{T}^2} \dive(2\nu(\theta)Du)\cdot \Delta u\,dx
=-\int_{\mathbb{T}^2} \mu\, \nabla\theta\cdot\Delta u\,dx,
\end{split}
\end{equation}
We first record that via H\"{o}lder's inequality and the interpolation inequality
\begin{equation}\label{3-71}
\begin{split}
&\int_{\mathbb{T}^2}  (u\cdot\nabla) u\cdot\Delta u\,dx  \leq \|\Delta u\|_{L^2}\|u\|_{L^4}\|\nabla u\|_{L^4} \leq C\|\Delta u\|_{L^2}^{\frac{3}{2}}\|u\|_{L^2}^{\frac{1}{2}}\|\nabla u\|_{L^2}
\\
&\leq \eta\|\Delta u\|^2_{L^2}+C_{\eta}\|u\|_{L^2}^2\|\nabla u\|^4_{L^2}
\end{split}
\end{equation}
for any positive constant $\eta$ (which will be determined later on), and
\begin{equation}\label{3-72}
\begin{split}
&-\int_{\mathbb{T}^2} \mu\, \nabla\theta\cdot\Delta u\,dx=\int _{\mathbb{T}^2} \theta\nabla\mu\cdot\Delta u\,dx\leq \eta\|\Delta u\|^2_{L^2}+C_{\eta}\|\theta\|^2_{{L^\infty}}\|\nabla\mu\|^2_{L^2}
\\
&\leq\eta\|\Delta u\|^2_{L^2}+C_{\eta}\|\nabla\mu\|^2_{L^2},
\end{split}
\end{equation}
where we have used the fact that $\dive\, u=0$ in the first equality. On the other hand, noting that
\begin{equation}\label{3-72-1}
\begin{split}
&\int_{\mathbb{T}^2} \dive(2\nu(\theta) Du)\cdot \Delta u\,dx=\int_{\mathbb{T}^2} 2\nu'(\theta) (Du\nabla\theta)\cdot \Delta u\,dx+\int_{\mathbb{T}^2} \nu(\theta)|\Delta u|^2\,dx,
\end{split}
\end{equation}
we find
\begin{equation}\label{3-73}
\begin{split}
&\int_{\mathbb{T}^2} \dive(2\nu(\theta) Du)\cdot \Delta u\,dx \geq\nu_1\|\Delta u\|^2_{L^2}-C\|\nabla u\|_{L^4}\|\nabla\theta\|_{L^4}\|\Delta u\|_{L^2}
\end{split}
\end{equation}
for some positive constant $\nu_1$, which, together with the inequality
\begin{equation}\label{3-73-1}
\begin{split}
\|\nabla\theta\|_{L^4}\|\nabla u\|_{L^4}\|\Delta u\|_{L^2} &\leq C \|\nabla\theta\|_{L^2}^{\frac{1}{2}}\|\Delta\theta\|_{L^2}^{\frac{1}{2}}\|\nabla u\|_{L^2}^{\frac{1}{2}}\|\Delta u\|_{L^2}^{\frac{3}{2}}\\
&\leq \eta\|\Delta u\|^2_{L^2}+C_{\eta}\|\nabla\theta\|_{L^2}^2\|\Delta\theta\|_{L^2}^2\|\nabla u\|_{L^2}^2,
\end{split}
\end{equation}
follows that
\begin{equation}\label{3-73-2}
\begin{split}
&\int_{\mathbb{T}^2} \dive(2\nu(\theta) Du)\cdot \Delta u\,dx \geq (\nu_1-\eta)\|\Delta u\|^2_{L^2}-C_{\eta}\|\nabla\theta\|_{L^2}^2\|\Delta\theta\|_{L^2}^2\|\nabla u\|_{L^2}^2.
\end{split}
\end{equation}
Inserting \eqref{3-71}-\eqref{3-73-2} into \eqref{3-70} yields
\begin{equation}\label{3-74}
\begin{split}
&\frac{1}{2}\frac{d}{dt}\|\nabla u\|^2_{L^2}+(\nu_1-3\eta)\|\Delta u\|^2_{L^2}\\
&\leq C_{\eta}\|\nabla\mu\|^2_{L^2}+C_{\eta}\|\nabla u\|_{L^2}^2(\|\nabla\theta\|_{L^2}^2\|\Delta\theta\|_{L^2}^2+\|u\|_{L^2}^2\|\nabla u\|^2_{L^2}).
\end{split}
\end{equation}
Taking $\eta=\frac{1}{8}\nu_1$ in \eqref{3-74} and using Gronwall' s inequality, we infer from \eqref{3-61} and \eqref{3-69} that
\begin{equation}\label{3-75}
\begin{split}
&\sup_{\tau \in [0, t]}\|\nabla u(\tau)\|^2_{L^2}+\nu_1\int^t_0\|\Delta u(\tau)\|^2_{L^2}\,d\tau
\\
&\leq  C(\|\nabla u_0\|^2_{L^2}+\int_0^t\|\nabla\mu\|^2_{L^2}\,d\tau)e^{C\int^t_0(\|\nabla\theta\|_{L^2}^2\|\Delta\theta\|^2_{L^2}+\|u\|_{L^2}^2\|\nabla u\|^2_{L^2})\,d\tau}\leq C_{T}.
\end{split}
\end{equation}
On the other hand, taking the $L^2$ inner product of the $u$-equation in \eqref{NSCH-2} with $\partial_t u$ ensures
\begin{equation*}
\begin{split}
&\|\partial_t u\|_{L^2}^2\leq C \|u\cdot\nabla  u\|_{L^2}^2+C\|\nabla \cdot(2\,\nu(\theta) D(u))\|_{L^2}^2+C|\int_{\mathbb{T}^2}\mu \nabla \theta\cdot \partial_t u\, dx|
\\
&=C \|u\cdot\nabla  u\|_{L^2}^2+C\|\nabla \cdot(2\,\nu(\theta) D(u))\|_{L^2}^2+C|\int_{\mathbb{T}^2}\theta \nabla \mu \cdot \partial_t u\, dx|\\
&\lesssim \|u\|_{L^4}^2\|\nabla  u\|_{L^4}^2+\|2\nu'(\theta) Du\nabla\theta+\nu(\theta)\Delta u\|_{L^2}^2+\|\theta \nabla \mu\|_{L^2}\|\partial_t u\|_{L^2}
\\
&\lesssim\|u\|_{L^2}\|\nabla  u\|_{L^2}^2\|\Delta u\|_{L^2}+\|\nabla u\|_{L^2}\|\Delta u\|_{L^2}\|\nabla \theta\|_{L^2}\|\Delta \theta\|_{L^2}+\|\Delta u\|_{L^2}^2+\| \nabla \mu\|_{L^2}\|\partial_t u\|_{L^2},
\end{split}
\end{equation*}
which follows
\begin{equation*}\label{3-80-1}
\begin{split}
&\|\partial_t u\|_{L^2}^2\lesssim \|u\|_{L^2}^2 \|\nabla  u\|_{L^2}^4+\| \nabla \mu\|_{L^2}^2 + (1+\|\nabla u\|_{L^2}^2+\|\nabla \theta\|_{L^2}^2) (\|\Delta u\|_{L^2}^2+\|\Delta \theta\|_{L^2}^2).
\end{split}
\end{equation*}
Therefore, we obtain from \eqref{3-69} and \eqref{3-75} again that
\begin{equation*}\label{3-80}
\|\partial_t u\|_{L^2([0, T]; L^2)} \leq C_{T},
\end{equation*}
which ends the proof of Lemma \ref{lem-unif-u-1}.
\end{proof}

\renewcommand{\theequation}{\thesection.\arabic{equation}}
\setcounter{equation}{0}
\section{Local well-posedness}\label{sect-4}

This section is devoted to the proof of the local well-posedness of the system \eqref{NSCH-2}. In order to achieve the goal, we need to prove the solution of the second equation in \eqref{NSCH-2} is separated from singular points of the non-linearity $\phi$.

\begin{thm}\label{thm-lp-nsch}
Under the assumptions in Theorem \ref{thm-glo-wp}, there exist $T>0$ and a unique solution $(u,\theta)$ on $[0,T]$ of the Navier-Stokes-Cahn-Hilliard
system \eqref{NSCH-2} such that
\begin{equation}\label{local-space}
\begin{split}
(u, \theta)\in &\bigg(C([0,T]; H^s(\mathbb{T}^2))\cap (L^2([0,T]; H^{s+1}(\mathbb{T}^2)) \bigg) \\
&\qquad \times \bigg(C([0,T]; H^s(\mathbb{T}^2))\cap L^2([0,T]; H^{s+2}(\mathbb{T}^2)))\bigg).
\end{split}
\end{equation}
Moreover, there holds
\begin{equation}\label{separate-cond-local-1}
\|\theta\|_{L^\infty([0, T] \times \mathbb{T}^2)}\leq 1-\frac{1}{4}\delta_0.
\end{equation}
\end{thm}

\begin{proof}
The proof is based on the energy method. We divide it into several steps.

\textbf{Step1: Construction of an approximate solution sequence}.

We shall first use the classical Friedrich's regularization
method to construct the approximate solutions to \eqref{NSCH-2}.
In order to do so, let us define the sequence of frequency cut-off operators $\suite P
n \N$ by
 \[
 P_{n}a \eqdefa \mathcal{F}^{-1}\bigl({\bf{1}}_{B(0,n)}\widehat a\bigr)
 \]
 and we define $(u_n,\, \theta_n)$ via
\begin{equation} \label{app-NSCH-2}
\begin{cases}
&\partial_t u_n+P_n \mathbb{P}(P_n u_n\cdot\nabla P_n u_n)-P_n \mathbb{P} \nabla \cdot\bigg(2\,\nu(P_n\theta_n) D(P_nu_n)\bigg)\\
&\qquad\qquad\qquad\qquad\qquad\qquad\qquad\qquad\qquad =-P_n \mathbb{P}(\mu_n \nabla P_n\theta_n),
\\
&\partial_t\theta_n+P_n(P_n u_n\cdot\nabla P_n\theta_n)=P_n\Delta \mu_n,
\\
&\nabla \cdot u_n=0,\quad \mu_n=P_n \phi(P_n\theta_n)-\Delta P_n\theta_n,
\\
&u_n(x,0)=P_n u_0(x), \quad \theta_n(x,0)=P_n \theta_0(x).
\end{cases}
\end{equation}
where ${\bf{1}}_{B(0,n)}$ is a characteristic function on the ball $B(0,n)$ centered at the origin with radius $n$ with $n \in \mathbb{N}$, and $\mathbb{P}$ denotes Leray's projection operator, which is given by $\mathbb{P}=(\delta_{jk}+R_jR_k)_{1 \leq j, \, k \leq 2}$
with Riesz transform $R_j$ defined by $\mathcal{F}(R_jf )(\xi) =\frac{i\xi_j}{|\xi|}\mathcal{F}(f)(\xi)$, and then $P_n \mathbb{P} = \mathbb{P}P_n$.

Without loss of generality, we restrict $n \geq n_0$ in what follows, where we choose the integer $n_0$ so large that
\begin{equation*} \label{restrict-n}
\bigg(\sum_{|m| \geq n_0} |m|^{-2s}\bigg)^{\frac{1}{2}} \bigg(\sum_{|m| \geq n_0} |m|^{2s}|\widehat{\theta}_0|^2\bigg)^{\frac{1}{2}}  \leq \frac{1}{4}\delta_0,
\end{equation*}
which implies that for any $n \geq n_0$
\begin{equation} \label{restrict-n-1}
\|P_{n}\theta_0-\theta_0\|_{L^{\infty}}\leq  \bigg(\sum_{|m| \geq n_0} |m|^{-2s}\bigg)^{\frac{1}{2}}  \bigg(\sum_{|m| \geq n_0} |m|^{2s}|\widehat{\theta}_0|^2\bigg)^{\frac{1}{2}} \leq \frac{1}{4}\delta_0.
\end{equation}

Because of properties of $L^2$ and $L^1$ functions the Fourier
transform of which are supported in the ball $B(0,n)$, the system \eqref{app-NSCH-2}
 appears to be an ordinary differential equation in the space
\begin{equation*}
 L^2_{n}\eqdefa \Bigl\{ a\in L^2(\mathbb{T}^2):\, \Supp \widehat a \subset B(0,n)\Bigr\}.
\end{equation*}
Then the Cauchy-Lipschitz theorem ensures that the
system \eqref{app-NSCH-2} has a unique solution $(u_n, \theta_n) \in C([0, T_n]; L^2(\mathbb{T}^2))$ for some $ T_n > 0$. Note that
$P_n^2=P_n$, $(P_nu_n, P_n\theta_n)$ is also a solution of \eqref{app-NSCH-2}. Thus the uniqueness of the solution implies that $(P_nu_n, P_n\theta_n)= (u_n, \theta_n)$ and the solution $(u_n, \theta_n)$ is smooth. Hence, the approximate system \eqref{app-NSCH-2} can be rewritten as
\begin{equation} \label{app-NSCH-3}
\begin{cases}
&\partial_t u_n+P_n \mathbb{P}(u_n\cdot\nabla  u_n)-P_n \mathbb{P} \nabla \cdot\bigg(2\,\nu(\theta_n) D(u_n)\bigg)=-P_n \mathbb{P}(\mu_n \nabla \theta_n),
\\
&\partial_t\theta_n+P_n(u_n\cdot\nabla \theta_n)=\Delta \mu_n,
\\
&\nabla \cdot u_n=0,\quad \mu_n=P_n\phi(\theta_n)-\Delta \theta_n,
\\
&u_n|_{t=0}=P_n u_0, \quad \theta_n|_{t=0}=P_n \theta_0.
\end{cases}
\end{equation}

\textbf{Step2:  Uniform estimates to the approximate solutions}\\
 Denote ${T}^{\ast}_n$ by the maximal existence time of the solution $(u_n, \theta_n)$, then, we first repeat the argument in the proof of Lemma \ref{lem-unif-u-1} to find
\begin{equation}\label{app-theta-bound-1}
\|\theta_n(t)\|_{L^\infty([0, T_n^\ast)\times \mathbb{T}^2)}\leq  1.
\end{equation}

Our goal in this step is to prove that there exists a positive time $0 < T< \inf_{n \in \mathbb{N}}{T}^{\ast}_n$ such that
\begin{equation}\label{app-bound-1}
\|\theta_n\|_{L^\infty([0, T]\times \mathbb{T}^2)}\leq  1-\frac{\delta_0}{4},
\end{equation}
and $(u_n, \theta_n)$ is uniformly bounded in the space
\begin{equation*}
\begin{split}
&\bigg(C([0,T]; H^s(\mathbb{T}^2))\cap (L^2([0,T]; H^{s+1}(\mathbb{T}^2)) \bigg) \\
&\qquad \times \bigg(C([0,T]; H^s(\mathbb{T}^2))\cap L^2([0,T]; H^{s+2}(\mathbb{T}^2)))\bigg).
\end{split}
\end{equation*}
The inequality \eqref{app-bound-1} allows us to reduce the singular equations \eqref{app-NSCH-3} (arising from the singular function $\phi$) to a regular problem, which plays a key role in what follows.

To obtain \eqref{app-bound-1}, we consider \eqref{app-NSCH-3} as a perturbation of its corresponding linear equations. For this, let's first define $u^n\eqdefa u^L_n+\bar u_n$, $\theta^n\eqdefa \theta^L_n+\bar\theta_n$, where $u^L_n(t)\eqdefa e^{t \nu(0)\Delta}P_n u_0$, $\theta^L_n(t)\eqdefa e^{-t \Delta^2}P_n\theta_0$. Then we may rewrite \eqref{app-NSCH-3} as the following $(\bar u^n, \bar\theta_n)$ equations
\begin{equation} \label{app-NSCH-4}
\begin{cases}
&\partial_t \bar u_n-P_n \mathbb{P} \nabla \cdot\bigg(2\,\nu(\theta^L_n+\bar\theta_n) D(\bar u_n)\bigg)=F_n(u^L_n, \theta_n^L, \bar u_n, \bar\theta_n),\\
&\partial_t\bar\theta_n+\Delta^2 \bar\theta_n=H_n(u^L_n, \theta_n^L, \bar u_n, \bar\theta_n),
\\
&\nabla \cdot \bar u_n=0,\\
&\bar u_n|_{t=0}=0, \quad \bar\theta_n|_{t=0}=0
\end{cases}
\end{equation}
with
\begin{equation}\label{3-f-h}
\begin{split}
&F_n(u^L_n, \theta_n^L, \bar u_n, \bar\theta_n):=-P_n \mathbb{P}\nabla\cdot (\bar u_n\otimes \bar u_n+ u^L_n\otimes \bar u_n)-P_n \mathbb{P} (\bar u_n\cdot \nabla  u^L_n+u^L_n\cdot\nabla u^L_n)\\
&\qquad -P_n \mathbb{P}((P_n\phi(\theta^L_n+\bar\theta_n)-\Delta \bar\theta_n) \nabla \theta^L_n)-P_n \mathbb{P}(P_n\phi(\theta^L_n+\bar\theta_n)\, \nabla \bar\theta_n)\\
&\qquad +P_n \mathbb{P} \nabla \cdot\bigg(2\,(\nu(\theta^L_n+\bar\theta_n)-\nu(0)) D(u^L_n)\bigg),\\
&H_n(u^L_n, \theta_n^L, \bar u_n, \bar\theta_n):=-P_n \nabla \cdot (\bar u_n\, \bar\theta_n+u^L_n\, \bar\theta_n+\bar u_n\,\theta^L_n+ u^L_n\, \theta^L_n)+\Delta P_n\phi(\theta^L_n+\bar\theta_n).
\end{split}
\end{equation}
From Remark \ref{rmk-bound-sep} and \eqref{restrict-n-1}, one can get, there is a positive time $T_1$ (independent of $n$) such that
\begin{equation}\label{app-bound-sepa-1}
\begin{split}
&\|\theta^L_n(t)\|_{L^\infty([0, T_1] \times \mathbb{T}^2)} \leq \|\theta^L_n(t)-P_n\theta_0\|_{L^\infty([0, T_1] \times \mathbb{T}^2)}+ \|P_{n}\theta_0\|_{L^\infty(\mathbb{T}^2)}\\
&\leq \|\theta^L_n(t)-P_n\theta_0\|_{L^\infty([0, T_1] \times \mathbb{T}^2)}+ \|P_{n}\theta_0-\theta_0\|_{L^\infty(\mathbb{T}^2)}+\|\theta_0\|_{L^\infty(\mathbb{T}^2)}\leq 1-\frac{1}{2}\delta_0
\end{split}
\end{equation}
for $\forall \, t\in [0, T_1]$.
Moreover, it is easy to find
\begin{equation}\label{3-56-1}
\begin{split}
&\|\theta^L_n\|_{L^\infty([0, +\infty); H^s(\mathbb{T}^2)}+\|\Delta\theta^L_n\|_{L^2([0, +\infty); H^s(\mathbb{T}^2)}\leq C\|\theta_0\|_{H^s},\\
&\|u^L_n\|_{L^\infty([0, +\infty); H^s(\mathbb{T}^2)}+\|\nabla u^L_n\|_{L^2([0, +\infty); H^s(\mathbb{T}^2)}\leq C\|u_0\|_{H^s}.
\end{split}
\end{equation}
Taking the $H^s(\mathbb{T}^2)$ inner product of the first equation of \eqref{app-NSCH-4} with $\bar u_n$, we have
\begin{equation}\label{3-82}
\begin{split}
&\frac{1}{2}\frac{d}{dt} \|\bar u_n\|_{H^s}^2 +\int_{\mathbb{T}^2} \Lambda^s \bigg( 2\nu(\theta^L_n+\bar\theta_n) D(\bar u_n)\bigg): \Lambda^s \nabla u_n \,dx =\int_{\mathbb{T}^2} \Lambda^s F_n\cdot \Lambda^s u_n \,dx.
\end{split}
\end{equation}
By using the commutator process, we find that
\begin{equation*}\label{3-83}
\begin{split}
&\int_{\mathbb{T}^2} \Lambda^s \bigg( 2\,\nu(\theta^L_n+\bar\theta_n) D(\bar u_n)\bigg): \Lambda^s \nabla \bar u_n \,dx
\\
&=\int_{\mathbb{T}^2}  2\,\nu(\theta^L_n+\bar\theta_n)D( \Lambda^s \bar u_n): \nabla \Lambda^s \bar u_n\, dx+2\int_{\mathbb{T}^2} [\Lambda^s,\nu(\theta^L_n+\bar\theta_n)]D(\bar u_n): \nabla \Lambda^s \bar u_n\,dx,
\end{split}
\end{equation*}
which along with Lemma \ref{lem-4} yields
\begin{equation}\label{3-83-0}
\begin{split}
&\int_{\mathbb{T}^2} \Lambda^s \bigg( 2\,\nu(\theta^L_n+\bar\theta_n) D(\bar u_n)\bigg): \Lambda^s \nabla \bar u_n \,dx
\\
&\geq \int_{\mathbb{T}^2}  \nu(\theta_n)|D( \Lambda^s \bar u_n)|^2dx
\\
&-C(\|\nu(\theta^L_n+\bar\theta_n)-\nu(0)\|_{H^s}\|D(\bar u_n)\|_{L^\infty}+\|\nabla \nu(\theta^L_n+\bar\theta_n)\|_{L^\infty}\| D(\bar u_n)\|_{H^{s-1}})\|\nabla\bar u_n\|_{H^s}\\
&\geq c\nu_1\|\nabla \bar u_n\|_{H^{s}}^2-C(1+\|\theta_n\|_{L^\infty})^{[s]+2}\\
&\qquad\qquad\qquad\times((\|\theta^L_n\|_{H^s}^2+\|\bar\theta_n\|_{H^s}^2) \|\nabla\bar u_n\|_{L^\infty}+ \|\nabla (\theta^L_n+\bar\theta_n)\|_{L^\infty}\|\bar u_n\|_{H^s})\|\nabla\bar u_n\|_{H^s}.
\end{split}
\end{equation}
Hence, from the interpolation inequality $\|\nabla f\|_{L^{\infty}(\mathbb{T}^2)} \leq C \|f\|_{H^1}^{\frac{s-1}{s}}\|\nabla f\|_{H^s}^{\frac{1}{s}}$, it leads to
\begin{equation*}\label{3-83-1}
\begin{split}
&\int_{\mathbb{T}^2} \Lambda^s \bigg( 2\,\nu(\theta^L_n+\bar\theta_n) D(\bar u_n)\bigg): \Lambda^s \nabla \bar u_n \,dx
\\
&\geq \frac{31}{32}c\nu_1\|\nabla u_n\|_{H^{s}}^2-C \|\bar u_n\|_{H^s}^2(\|\nabla \theta^L_n\|_{H^s}^2+ \|\nabla\bar \theta_n\|_{H^s}^2)-C\|\bar\theta_n\|_{H^s}^2 \|\nabla\bar u_n\|_{H^s}\\
&\qquad \qquad\qquad\qquad-C\|\theta^L_n\|_{H^s}^2 \|\bar u_n\|_{H^1}^{\frac{s-1}{s}}\|\nabla\bar u_n\|_{H^s}^{\frac{1}{s}+1}\\
&\geq \frac{15}{16}c\nu_1\|\nabla u_n\|_{H^{s}}^2-C \|\bar u_n\|_{H^s}^2(\|\nabla \theta^L_n\|_{H^s}^2+ \|\nabla\bar \theta_n\|_{H^s}^2+\|\theta^L_n\|_{H^s}^{\frac{4s}{s-1}})-C\|\bar\theta_n\|_{H^s}^4.
\end{split}
\end{equation*}
Hence, we obtain
\begin{equation}\label{3-83-2}
\begin{split}
&\int_{\mathbb{T}^2} \Lambda^s \bigg( 2\,\nu(\theta^L_n+\bar\theta_n) D(\bar u_n)\bigg): \Lambda^s \nabla \bar u_n \,dx
\\
&\geq \frac{15}{16}c\nu_1\|\nabla u_n\|_{H^{s}}^2-C \|\bar u_n\|_{H^s}^2(\|\nabla \theta^L_n\|_{H^s}^2+ \|\nabla\bar \theta_n\|_{H^s}^2+1)-C\|\bar\theta_n\|_{H^s}^4.
\end{split}
\end{equation}
Let's now estimate $\int\Lambda^s F_n\cdot \Lambda^s u_n \,dx$. We will split it into six terms by the definition of $F_n$ in \eqref{3-f-h}, and then bound them step by step.
Thanks to H\"{o}lder's inequality and integration by parts, one can get
\begin{equation}\label{3-84}
\begin{split}
&|\int_{\mathbb{T}^2} \Lambda^s\nabla\cdot(\bar u_n \otimes \bar u_n )\cdot \Lambda^s \bar u_n \,dx|\leq \|\bar u_n \otimes  \bar u_n\|_{H^s}\|\nabla \bar u_n\|_{H^s}\leq \frac{c\nu_1}{16}\|\nabla \bar u_n\|_{H^{s}}^2+C\|\bar u_n\|_{H^s}^4,
\end{split}
\end{equation}
\begin{equation}\label{3-84-1}
\begin{split}
&|\int_{\mathbb{T}^2} \Lambda^s\nabla\cdot(u^L_n \otimes \bar u_n )\cdot \Lambda^s \bar u_n \,dx|\leq \frac{c\nu_1}{16}\|\nabla \bar u_n\|_{H^{s}}^2+C\|\bar u_n\|_{H^s}^2\|u^L_n\|_{H^s}^2,
\end{split}
\end{equation}
and
\begin{equation}\label{3-84-2}
\begin{split}
&|\int_{\mathbb{T}^2} \Lambda^s(\bar u_n\cdot \nabla  u^L_n+u^L_n\cdot\nabla u^L_n)\cdot \Lambda^s \bar u_n \,dx|\leq \|\nabla  u^L_n\|_{H^{s}}^2+ C(\|\bar u_n\|_{H^s}^2+\|u^L_n\|_{H^{s}}^2)\|\bar u_n\|_{H^s}^2.
\end{split}
\end{equation}
By using Lemma \ref{lem-4}, we obtain
\begin{equation}\label{3-84-3}
\begin{split}
&|\int_{\mathbb{T}^2} \Lambda^s((P_n\phi(\theta^L_n+\bar\theta_n)-\Delta \bar\theta_n) \nabla \theta^L_n)\cdot \Lambda^s \bar u_n \,dx|\\
&\leq \|\nabla \theta^L_n\|_{H^s}\|\bar u_n\|_{H^s}(\|\phi(\theta^L_n+\bar\theta_n)\|_{H^s}+\|\Delta\bar\theta_n\|_{H^s})\\
& \leq \|\nabla  u^L_n\|_{H^{s}}^2+ C(\|\bar \theta_n\|_{H^s}^2+\|\theta^L_n\|_{H^{s}}^2+\|\Delta\bar\theta_n\|_{H^s}^2)\|\bar u_n\|_{H^s}^2,
\end{split}
\end{equation}
\begin{equation}\label{3-84-4}
\begin{split}
&|\int_{\mathbb{T}^2} \Lambda^s(P_n\phi(\theta^L_n+\bar\theta_n) \nabla \bar\theta_n)\cdot \Lambda^s \bar u_n \,dx|\leq \|\nabla \bar\theta_n\|_{H^s}\|\bar u_n\|_{H^s}\|\phi(\theta^L_n+\bar\theta_n)\|_{H^s}\\
& \leq \|\nabla  \bar\theta_n\|_{H^{s}}^2+ C (\|\bar \theta_n\|_{H^s}^2+\|\theta^L_n\|_{H^{s}}^2)\|\bar u_n\|_{H^s}^2,
\end{split}
\end{equation}
and
\begin{equation}\label{3-84-5}
\begin{split}
&|\int_{\mathbb{T}^2} \Lambda^s\nabla \cdot\bigg(2\,(\nu(\theta^L_n+\bar\theta_n)-\nu(0)) D(u^L_n)\bigg) \cdot \Lambda^s \bar u_n \,dx|\\
&\leq \frac{1}{16}c\nu_1\|\nabla \bar u_n\|_{H^{s}}^2+ C\|(\nu(\theta^L_n+\bar\theta_n)-\nu(0))\|_{H^s}^2\| \nabla u^L_n\|_{H^s}^2\\
&\leq \frac{1}{16}c\nu_1\|\nabla \bar u_n\|_{H^{s}}^2+ C(\|\theta^L_n\|_{H^s}^2+\|\bar\theta_n\|_{H^s}^2)\| \nabla u^L_n\|_{H^s}^2.
\end{split}
\end{equation}
Thus, it follows from \eqref{3-84}-\eqref{3-84-5} that
\begin{equation}\label{3-84-6}
\begin{split}
&|\int_{\mathbb{T}^2} \Lambda^s F_n\cdot \Lambda^s u_n \,dx|\leq \frac{3}{16}c\nu_1\|\nabla \bar u_n\|_{H^{s}}^2+\|\nabla  \bar\theta_n\|_{H^{s}}^2+ C(\|\bar u_n\|_{H^s}^2+\|u^L_n\|_{H^{s}}^2)\|\bar u_n\|_{H^s}^2\\
&\quad + C(\|\bar \theta_n\|_{H^s}^2+\|\theta^L_n\|_{H^{s}}^2+\|\Delta\bar\theta_n\|_{H^s}^2)\|\bar u_n\|_{H^s}^2+ C(1+\|\theta^L_n\|_{H^s}^2+\|\bar\theta_n\|_{H^s}^2)\| \nabla u^L_n\|_{H^s}^2.
\end{split}
\end{equation}
Inserting \eqref{3-83-2} and \eqref{3-84-6} into \eqref{3-82} ensures
\begin{equation}\label{3-85}
\begin{split}
&\frac{d}{dt} \|\bar u_n\|_{H^s}^2+\frac{3}{2} c\nu_1\|\nabla\bar u_n\|_{H^{s}}^2\\
&\leq C(\|\bar u_n\|_{H^s}^2+\|\bar\theta_n\|_{H^s}^2)(\| \nabla u^L_n\|_{H^s}^2+\|\nabla \theta^L_n\|_{H^s}^2+1)\\
&\qquad+ C\|\bar u_n\|_{H^s}^2(\|\bar u_n\|_{H^s}^2+\|\bar \theta_n\|_{H^s}^2+\|\Delta\bar\theta_n\|_{H^s}^2) +C\|\bar\theta_n\|_{H^s}^4+C\| \nabla u^L_n\|_{H^s}^2
\end{split}
\end{equation}
On the other hand, applying the operator $\Lambda^s$ to the second equation in \eqref{app-NSCH-4}, we obtain
\begin{equation}\label{3-86}
\begin{split}
\partial_t\Lambda^s\bar \theta_n+\Delta^2\Lambda^s\bar\theta_n=\Lambda^s H_n.
\end{split}
\end{equation}
Taking the $L^2$ inner product with $\Lambda^s\theta_n$, we may obtain
\begin{equation}\label{3-87}
\begin{split}
&\frac{1}{2}\frac{d}{dt}\|\bar\theta_n\|^2_{H^s}+\|\Delta\bar\theta\|^2_{H^s}=\int_{\mathbb{T}^2}  \Lambda^sH_n\, \Lambda^s\bar\theta_n\, dx.
\end{split}
\end{equation}
In order to estimate $\int \Lambda^sH_n\, \Lambda^s\bar\theta_n\, dx$, we deduce, according to the definition of $H_n$, that
\begin{equation*}
\begin{split}
&|\int_{\mathbb{T}^2} \Lambda^s\nabla\cdot(\bar u_n\, \bar\theta_n+u^L_n\, \bar\theta_n+\bar u_n\,\theta^L_n+ u^L_n\, \theta^L_n)\cdot \Lambda^s \bar \theta_n \,dx|\leq \|\bar u_n \otimes  \bar u_n\|_{H^s}\|\nabla \bar u_n\|_{H^s}\\
& \leq \|\nabla \bar \theta_n\|_{H^{s}}^2+ C(\|\bar u_n\|_{H^s}^2+\|u^L_n\|_{H^s}^2)(\|\bar \theta_n\|_{H^s}^2+\|\theta^L_n\|_{H^s}^2)
\end{split}
\end{equation*}
and
\begin{equation*}
\begin{split}
&|\int_{\mathbb{T}^2}  \Delta\phi(\theta^L_n+\bar\theta_n)\,\Lambda^s\bar\theta_n\, dx|=|\int_{\mathbb{T}^2}  \phi(\theta^L_n+\bar\theta_n)\,\Lambda^s\Delta\bar\theta_n\, dx|\\
&\leq  \frac{1}{4}\|\Delta\bar \theta_n\|_{H^{s}}^2+ C\|\phi(\theta^L_n+\bar\theta_n)\|_{H^s}\leq  \frac{1}{4}\|\Delta\bar \theta_n\|_{H^{s}}^2+ C(\|\bar\theta_n\|_{H^s}^2+\|\theta^L_n\|_{H^s}^2).
\end{split}
\end{equation*}
From these, it follows from \eqref{3-56-1} and the interpolation inequality
\begin{equation}\label{interpo-ineq-1}
\|\nabla \bar \theta_n\|_{H^{s}}^2\leq \frac{1}{8}\|\Delta\bar \theta_n\|_{H^{s}}^2+C\|\bar \theta_n\|_{H^{s}}^2
\end{equation}
that
\begin{equation}\label{3-87-2}
\begin{split}
&|\int_{\mathbb{T}^2}  \Lambda^sH_n\, \Lambda^s\bar\theta_n\, dx|\leq   \frac{3}{8}\|\Delta\bar \theta_n\|_{H^{s}}^2+ C(\|\bar u_n\|_{H^s}^2+1)(\|\bar \theta_n\|_{H^s}^2+1).
\end{split}
\end{equation}
Inserting \eqref{3-87-2} into \eqref{3-87} leads to
\begin{equation}\label{3-88}
\begin{split}
&\frac{d}{dt}\|\bar\theta_n\|_{H^s}^2+\frac{5}{4}\|\Delta\bar\theta\|_{H^s}^2\leq  C(1+\|\bar u_n\|_{H^s}^2)(1+\|\bar \theta_n\|_{H^s}^2),
\end{split}
\end{equation}
which along with \eqref{3-85} and \eqref{interpo-ineq-1} gives rise to
\begin{equation}\label{3-89}
\begin{split}
&\frac{d}{dt} (\|\bar u_n\|_{H^s}^2+\|\bar\theta_n\|_{H^s}^2)+ c\nu_1\|\nabla\bar u_n\|_{H^{s}}^2+\|\Delta\bar\theta_n\|_{H^s}^2\\
&\leq C(\|\bar u_n\|_{H^s}^2+\|\bar\theta_n\|_{H^s}^2)(\| \nabla u^L_n\|_{H^s}^2+\|\nabla \theta^L_n\|_{H^s}^2+1)\\
&\qquad+ C\|\bar u_n\|_{H^s}^2(\|\bar u_n\|_{H^s}^2+\|\bar \theta_n\|_{H^s}^2+\|\Delta\bar\theta_n\|_{H^s}^2) +C\|\bar\theta_n\|_{H^s}^4+C(1+\| \nabla u^L_n\|_{H^s}^2).
\end{split}
\end{equation}
Therefore, defining
\begin{equation*}\label{3-89-1}
\begin{split}
&E_n(t) := \sup_{\tau \in [0, t]} (\|u_n(\tau)\|_{H^s}^2+\|\theta_n(\tau)\|_{H^s}^2)+\|\nabla\bar u_n\|_{L^2([0, t]; H^{s})}^2+\|\Delta\bar\theta_n\|_{L^2([0, t]; H^{s})}^2,
\end{split}
\end{equation*}
it follows
\begin{equation}\label{3-89-2}
\begin{split}
&E_n(t)\leq C_0\,E_n(t)(\| \nabla u^L_n\|_{L^2([0, t]; H^{s})}^2+\|\nabla \theta^L_n\|_{L^2([0, t]; H^{s})}^2+t)\\
&\qquad \qquad\qquad + C_0\,E_n(t)^2 +C_0(\| \nabla u^L_n\|_{L^2([0, t]; H^{s})}^2+t)
\end{split}
\end{equation}
with the constant $C_0 \geq 1+\beta^2$, where the constant $\beta$ satisfies $\|f\|_{L^{\infty}(\mathbb{T}^2)} \leq \beta \|f\|_{H^s(\mathbb{T}^2)}$ (for $\forall \, f \in H^s(\mathbb{T}^2)$ with $s>1$).

Note that, from \eqref{3-56-1}, there exists a position time $T_0 \in (0, \min\{T_1, \frac{\delta_0^2}{2^7C_0^2}\}]$ (with $T_1$ in Remark \ref{rmk-bound-sep}) so small that
\begin{equation}\label{3-89-3}
\begin{split}
&\| \nabla u^L_n\|_{L^2([0, T_0]; H^{s})}^2+\|\nabla \theta^L_n\|_{L^2([0, T_0]; H^{s})}^2\leq \frac{\delta_0^2}{2^7C_0^2}.
\end{split}
\end{equation}

Therefore, by the bootstrap argument, we may claim that there is a positive time $T=T(u_0,\theta_0)(\leq \min\{T_0,\,T^\ast_n\})$ independent of $n$ such that for all $n$,
\begin{equation} \label{3-90-5}
\begin{split}
&\sup_{0\leq t\leq T} E_n(t)\leq \frac{\delta_0^2}{16C_0}.
\end{split}
\end{equation}
In effect, we first denote ${T}^{\ast}_n$ by the maximal existence time of the solution $(u_n, \theta_n)$ and define
\begin{equation} \label{3-90-1}
\widetilde{T}^{\ast}_n:=\sup \{t \in [0, {T}^{\ast}_n):\, \sup_{0\leq \tau \leq t}E_n(\tau)\leq \frac{\delta_0^2}{16C_0} \},
\end{equation}
From \eqref{3-89-2}, we find that for any $t \in [0, \widetilde{T}^{\ast}_n)$
\begin{equation*}\label{3-90-2}
\begin{split}
&E_n(t)\leq \frac{\delta_0}{2^6C_0}(\| \nabla u^L_n\|_{L^2([0, t]; H^{s})}^2+\|\nabla \theta^L_n\|_{L^2([0, t]; H^{s})}^2+t)\\
&\qquad \qquad\qquad + \frac{\delta_0}{16}\,E_n(t) +C_0(\| \nabla u^L_n\|_{L^2([0, t]; H^{s})}^2+t),
\end{split}
\end{equation*}
and then
\begin{equation*}\label{3-90-3}
\begin{split}
&E_n(t)\leq \frac{1}{60}\delta_0(\| \nabla u^L_n\|_{L^2([0, t]; H^{s})}^2+\|\nabla \theta^L_n\|_{L^2([0, t]; H^{s})}^2+t) +\frac{16}{15}C_0(\| \nabla u^L_n\|_{L^2([0, t]; H^{s})}^2+t),
\end{split}
\end{equation*}
which along with \eqref{3-89-3} implies that $\widetilde{T}^{\ast}_n \geq T_0$. Otherwise, we have
\begin{equation*}\label{3-90-4}
\begin{split}
&E_n(t)\leq \delta_0\frac{\delta_0^2}{32C_0^2} +2C_0\frac{\delta_0^2}{32C_0^2} \leq \frac{3}{4}\frac{\delta_0^2}{16C_0} , \quad \forall \, t \in [0, \widetilde{T}^{\ast}_n),
\end{split}
\end{equation*}
which contradicts with the definition of $\widetilde{T}^{\ast}_n $, \eqref{3-90-1}. Thus, the approximate solution $(u_n, \theta_n)$ exists on
$[0, T_0]$ and satisfies \eqref{3-90-5},
which along with \eqref{3-89} implies that
\begin{equation}\label{3-90-6}
\begin{split}
&\sup_{t \in [0, T_0]} \|\bar\theta_n(t)\|_{L^\infty}\leq \beta E_n(T_0)^{\frac{1}{2}}\leq \frac{\delta_0}{4}.
\end{split}
\end{equation}
Therefore, combining \eqref{3-90-6} with \eqref{separate-condition-1} and \eqref{app-bound-sepa-1} yields \eqref{app-bound-1}.

Moreover, we get from the first equations in \eqref{app-NSCH-3} that
\begin{equation*} \label{3-90-7}
\begin{split}
&\|\partial_t u_n\|_{H^{s-1}}\leq \|u_n\cdot\nabla  u_n\|_{H^{s-1}}+\|\nabla \cdot(2\,\nu(\theta_n) D(u_n))\|_{H^{s-1}}+\|\mu_n \nabla \theta_n\|_{H^{s-1}}
\\
&\leq C(\|u_n\|_{H^{s}}^2+\|\nabla u_n\|_{H^{s}}+\|\nu(\theta_n)-\nu(0)\|_{L^{\infty}} \|\nabla u_n\|_{H^{s}}\\
&\qquad+\|\nabla u_n\|_{L^{\infty}} \|\nu(\theta_n)-\nu(0)\|_{H^{s}}+\|\phi(\theta_n)\|_{L^{\infty}}\| \nabla \theta_n\|_{H^{s-1}}+\|\phi(\theta_n) \|_{H^{s-1}}\|\nabla \theta_n\|_{L^{\infty}}\\
&\qquad+\|\Delta \theta_n \|_{L^{\infty}}\|\nabla \theta_n\|_{H^{s-1}}+\|\nabla \theta_n\|_{L^{\infty}}\|\Delta \theta_n \|_{H^{s-1}})
\\
&\leq C(\|u_n\|_{H^{s}}^2+\|\nabla u_n\|_{H^{s}}+\|\nabla u_n\|_{H^{s}} \|\theta_n\|_{H^{s}}+\| \nabla \theta_n\|_{H^{s}}\\
&\qquad+\|\theta_n\|_{H^{s}}\|\nabla \theta_n\|_{H^s}+\|\Delta \theta_n \|_{H^s}\|\theta_n\|_{H^{s}}+\|\theta_n\|_{H^s}\|\Delta \theta_n \|_{H^{s}}),
\end{split}
\end{equation*}
which along with \eqref{3-90-6} gives rise to
\begin{equation*}\label{3-90-7}
\{\partial_t u_n\}_{n \in \mathbb{N}}\quad\mbox{is uniformly bounded in} \quad L^2([0,T];H^{s-1}(\mathbb{T}^2)).
\end{equation*}
On the other hand, taking the $H^{s-2}$ inner product of the second equations in \eqref{app-NSCH-3} with $\partial_t\theta_n$ gives rise to
\begin{equation*}\label{3-90-8}
\begin{split}
&\|\partial_t\theta_n\|_{H^{s-2}}\leq  \|u_n\cdot\nabla \theta_n\|_{H^{s-2}}+\|\Delta \phi(\theta_n)\|_{H^{s-2}}+\|\Delta^2\theta_n\|_{H^{s-2}}
\\
&\leq C(\|u_n\|_{H^{s}}\|\nabla \theta_n\|_{H^{s}}+\|\theta_n\|_{H^{s}}+\|\Delta\theta_n\|_{H^{s}}),
\end{split}
\end{equation*}
which results from \eqref{3-56-1} and \eqref{3-90-5} that
\begin{equation*}\label{3-90-10}
\{\partial_t \theta_n\}_{n \in \mathbb{N}}\quad\mbox{is uniformly bounded in} \quad L^2([0,T];H^{s-2}(\mathbb{T}^2)).
\end{equation*}
Therefore, we obtain
\begin{equation}\label{3-91}
\begin{split}
&\{(u_n, \theta_n)\}_{n \in \mathbb{N}}\quad\mbox{is uniformly bounded in} \quad \mathbb{C}([0,T];H^s(\mathbb{T}^2)),\\
&\{(\nabla u_n, \Delta\theta_n)\}_{n \in \mathbb{N}} \quad\hbox{is uniformly bounded in} \quad L^2([0,T];H^s(\mathbb{T}^2)),\\
&\{\partial_t u_n\}_{n \in \mathbb{N}}\quad\mbox{is uniformly bounded in} \quad L^2([0,T];H^{s-1}(\mathbb{T}^2)),\\
&\{\partial_t \theta_n\}_{n \in \mathbb{N}}\quad\mbox{is uniformly bounded in} \quad L^2([0,T];H^{s-2}(\mathbb{T}^2)).
\end{split}
\end{equation}

\textbf{Step 4: Convergence}

Thanks to \eqref{3-91} and Aubin-Lions's compactness theorem, there exists a subsequence of $\{(u_n, \theta_n)\}_{\varepsilon>0}$, still denoted by  $\{(u_n, \theta_n)\}_{\varepsilon>0}$, which converges
to some function $(u, \theta)\in \bigg(L^{\infty}([0, T]; H^s(\mathbb{T}^2)) \cap L^{2}([0, T]; H^{s+1}(\mathbb{T}^2))\bigg) \times \bigg(L^{\infty}([0, T]; H^s(\mathbb{T}^2)) \cap L^{2}([0, T]; H^{s+2}(\mathbb{T}^2))\bigg)$ such that
\begin{equation*}\label{3-92}
\begin{split}
&u_n \rightarrow u \quad \mbox{in} \quad L^{2}([0, T]; H^{s+1}(\mathbb{T}^2)),\\
&\theta_n \rightarrow \theta \quad \mbox{in} \quad L^{2}([0, T]; H^{s+2}(\mathbb{T}^2)).
\end{split}
\end{equation*}
Then passing to limit in \eqref{app-NSCH-3}, it is easy to see that $(u, \, \theta)$ satisfies \eqref{NSCH-2} in the weak sense. Moreover, there hold \eqref{separate-cond-local-1} from \eqref{app-bound-1} and
\begin{equation}\label{3-93}
\begin{split}
&\|(u,\,\theta)\|_{L^{\infty}([0, T]; H^s(\mathbb{T}^2)}+ \|(\nabla u, \, \Delta\theta)\|_{L^{2}([0, T]; H^{s})}\\
&\qquad\qquad\qquad\qquad\qquad +\|\partial_t u\|_{L^{2}([0, T]; H^{s-1})}+\|\partial_t \theta\|_{L^{2}([0, T]; H^{s-2})} \leq C.
\end{split}
\end{equation}
In order to get the continuity in time of the solution, we need the refine estimate of $(u, \, \theta)$. For this, we denote (see Appendix)
\begin{equation*}
\begin{split}
&\|f\|_{\widetilde{L}^{\infty}([0, t]; H^s)}:= \|2^{qs}\|\Delta_qf\|_{L^{\infty}[[0, t]; L^2]}\|_{\ell^2(q \in \mathbb{N}\cup\{-1\})},\\
 &\|f\|_{\widetilde{L}^{1}([0, t]; H^s)}:= \|2^{qs}\|\Delta_qf\|_{L^{1}[[0, t]; L^2]}\|_{\ell^2(q \in \mathbb{N}\cup\{-1\})}.
 \end{split}
\end{equation*}
Apply the operator $\Delta_q$ with $q \in \mathbb{N}\cup\{-1\}$ to the $u$ equations in \eqref{NSCH-2} to find
\begin{equation*}
\begin{split}
&\partial_t \Delta_{q}u+\Delta_{q}(u\cdot\nabla  u)+\nabla \Delta_{q} g- \dive\,(2\nu(\theta) D(\Delta_{q}u))\\
&\qquad\qquad\qquad\qquad\qquad =\dive\,(2[\nu(\theta), \Delta_{q}] D(u))+\Delta_{q}(\phi(\theta) \nabla \theta)-\Delta_{q}(\Delta \theta \nabla \theta).
\end{split}
\end{equation*}
Denote $u^H:=u-{\Delta}_{-1} u$, then one may get
\begin{equation*}
\begin{split}
&\frac{1}{2} \frac{d}{dt} \|{\Delta}_{q} u\|_{L^2}^2+c\nu_1 2^{2q}\|{\Delta}_{q} u^H\|_{L^2}^2 \leq  C \|{\Delta}_{q} u\|_{L^2}\Bigl(\|\Delta_{q}(u\cdot\nabla  u)\|_{L^2}\\
&\qquad\qquad +
2^q\|[\nu(\theta)-\nu(0), \Delta_{q}] D(u)\|_{L^2}+\|\Delta_{q}(\phi(\theta) \nabla \theta)\|_{L^2}+\|\Delta_{q}(\Delta \theta \nabla \theta)\|_{L^2}\|_{L^2}\bigg),
\end{split}
\end{equation*}
which follows that
\begin{equation*}
\begin{split}
&\|{\Delta}_{q} u\|_{L^{\infty}_t(L^2)} +c\nu_1 2^{2q}\|{\Delta}_{q} u^H\|_{L^1_t(L^2)}\leq \|{\Delta}_{q} u_0\|_{L^2} + C \Bigl(\|\Delta_{q}(u\cdot\nabla  u)\|_{L^{1}_t(L^2)}\\
&\qquad +
2^q\|[\nu(\theta)-\nu(0), \Delta_{q}] D(u)\|_{L^{1}_t(L^2)}+\|\Delta_{q}(\phi(\theta) \nabla \theta)\|_{L^{1}_t(L^2)}+\|\Delta_{q}(\Delta \theta \nabla \theta)\|_{L^2}\|_{L^{1}_t(L^2)}\bigg),
\end{split}
\end{equation*}
Thanks to Lemma \ref{lem-pressure-1},
\begin{equation*}
\begin{split}
&\|([\nu(\theta)-\nu(0), \Delta_{q}] D(u))\|_{L^2}^2 \\
&\leq Cc_{q}(t)2^{-q(s+1)}(\|\nu(\theta)-\nu(0)\|_{H^{s+2}}\|u\|_{H^{1}}+\|\nu(\theta)-\nu(0)\|_{H^{2}}\|u\|_{H^{s+1}})
\end{split}
\end{equation*}
with $\sum_{q \geq -1}c_{q}(t)^2 \leq 1$, we may obtain
\begin{equation*}
\begin{split}
&\|u\|_{\widetilde{L}^{\infty}_t(H^s)} +c\|u\|_{\widetilde{L}^{1}_t(H^{s+2})}\leq \|u_0\|_{H^s} + \|u\|_{L^1_t(L^2)}\\
&\,+C\int_0^t \bigg(\|u\cdot \nabla u\|_{H^s}+ \|\nu(\theta)-\nu(0)\|_{H^{s+2}}\|u\|_{H^{1}}\\
&\qquad\qquad\qquad\qquad+\|\nu(\theta)-\nu(0)\|_{H^{2}}\|u\|_{H^{s+1}}+
\|\phi(\theta) \nabla \theta\|_{H^s}+\|\Delta \theta \nabla \theta\|_{H^s}\bigg)\, d\tau.
\end{split}
\end{equation*}
Therefore, it follows from \eqref{separate-cond-local-1}, Lemma \ref{lem-4}, and \eqref{3-93} that
\begin{equation}\label{3-96-10}
\begin{split}
&\|u\|_{\widetilde{L}^{\infty}_T(H^s)}+c\|u\|_{\widetilde{L}^{1}_T(H^{s+2})} \\
&\leq C_T\|u_0\|_{H^s} + C \Bigl(\|u\|_{{L}^{\infty}_T(H^s)}\|\nabla  u\|_{{L}^{1}_T(H^s)}+\|\theta\|_{{L}^{2}_T(H^{s+2})}) \| u\|_{{L}^{2}_T(H^{s+1})}\\
&\qquad\qquad\qquad\qquad\qquad+\|\theta\|_{{L}^{\infty}_T(H^s)} \|\nabla \theta\|_{{L}^{1}_T(H^s)}+\|\Delta \theta\|_{{L}^{2}_T(H^s)}\| \nabla \theta\|_{{L}^{2}_T(H^s)}\bigg)\\
&\leq C(T, u_0, \theta_0).
\end{split}
\end{equation}
Similarly, we may get
\begin{equation*}
\begin{split}
&\|\theta\|_{\widetilde{L}^{\infty}_T(H^s)}+\|\theta\|_{\widetilde{L}^{1}_T(H^{s+4})} \leq C(T, u_0, \theta_0).
\end{split}
\end{equation*}

\textbf{Step 5: Continuity in time of the solution}

Let's now prove the continuity in time of the solution. Indeed, from \eqref{3-96-10}, $\|u\|_{\widetilde{L}^{\infty}_T(H^s)}<C$, so, for any $\varepsilon>0$, one can take $N \in \mathbb{N}$ large enough such that
\begin{equation}\label{3-96-11}
\begin{split}
&\sum_{q \geq N}2^{2qs}\|{\Delta}_{q} u\|_{L^{\infty}_T(L^2)}^2 < \frac{\varepsilon}{4}.
\end{split}
\end{equation}
For any $t \in [0, T]$ and $h$ such that $t +h \in [0, T]$, we deduce from \eqref{3-96-11} and \eqref{3-93} that
\begin{equation*}\label{3-96-12}
\begin{split}
&\|u(t +h)-u(t)\|_{H^s}^2\leq \sum_{-1\leq q \leq N}2^{2qs}\|{\Delta}_{q} (u(t +h)-u(t))\|_{L^2}^2 +\frac{\varepsilon}{2}\\
&\leq \sum_{-1\leq q \leq N}2^{2q} 2^{2q(s-1)}\|{\Delta}_{q} (u(t +h)-u(t))\|_{L^2}^2 +\frac{\varepsilon}{2}\\
&\leq \sum_{-1\leq q \leq N}2^{2q}\|u(t +h)-u(t)\|_{H^{s-1}}^2 +\frac{\varepsilon}{2}\leq \sum_{-1\leq q \leq N}2^{2q} |h|\|\partial_t u\|_{L^2_T(H^{s-1})}^2 +\frac{\varepsilon}{2}\\
&\leq (2+N)2^{2 N} \|\partial_t u\|_{L^2_T(H^{s-1})}^2 \, |h| +\frac{\varepsilon}{2}< \varepsilon
\end{split}
\end{equation*}
for $|h|$ small enough. Hence, $u(t)$ is continuous in $H^s(\mathbb{T}^2)$ for any time $t \in [0, T]$. Similarly, we may get that $\theta(t)$ is also continuous in $H^s(\mathbb{T}^2)$ for any time $t \in [0, T]$.

\textbf{Step 6:  Continuous dependency and uniqueness of the solution}

Let $(u^1, \theta^1)$ and $(u^2,\theta^2)$ be two solutions of \eqref{NSCH-2} and satisfy \eqref{local-space} and \eqref{separate-cond-local-1}.
We denote $\tilde{u}=u^1-u^2, \tilde{\theta}=\theta^1-\theta^2$, and $\tilde{g}=g^1-g^2$. Then $(\tilde{u},\tilde{\theta})$ satisfies
\begin{equation*}
\begin{cases}
&\partial_t\tilde{u}+(\tilde{u}\cdot\nabla)u^1+(u^2\cdot\nabla)\tilde{u}-\dive(\nu(\theta^1)2D\tilde{u})-\dive((\nu(\theta^1)-\nu(\theta^2))2Du^2)+\nabla\tilde{g}
\\
&\qquad\qquad\qquad\qquad=-\Delta\theta^1\nabla\tilde{\theta}-\Delta\tilde{\theta}\nabla\theta^2+\nabla\tilde{\theta}\phi(\theta^1)+\nabla\theta^2(\phi(\theta^1)-\phi(\theta^2)),
\\
&\partial_t\tilde{\theta}+(\tilde{u}\cdot\nabla)\theta^1+(u^2\cdot\nabla)\tilde{\theta}=-\Delta^2\tilde{\theta}+\Delta(\phi(\theta^1)-\phi(\theta^2)).
\end{cases}
\end{equation*}
Taking $H^s(\mathbb{T}^2)$ energy estimate yields
\begin{equation}\label{3-95}
\begin{split}
&\frac{1}{2}\frac{d}{dt}\|(\tilde{u},\tilde{\theta})\|_{H^s}^2+c\nu_1\|\nabla\tilde{u}\|^2_{H^s}+\|\Delta\tilde{\theta}\|^2_{H^s}
\\
&\leq \int_{\mathbb{T}^2}  [\nu(\theta^1)-\nu(0), \Lambda^s]D\tilde{u}:\nabla \Lambda^s\tilde{u}\,dx-\int_{\mathbb{T}^2} \Lambda^s(\tilde{u}\cdot\nabla u^1+u^2\cdot\nabla \tilde{u})\cdot\Lambda^s\tilde{u}\,dx\\
&-\int_{\mathbb{T}^2} \Lambda^s((\nu(\theta^1)-\nu(\theta^2))2Du^2):\nabla\Lambda^s\tilde{u}\,dx-\int_{\mathbb{T}^2} \Lambda^s(\tilde{u}\cdot\nabla\theta^1+u^2\cdot\nabla\tilde{\theta})\Lambda^s\tilde{\theta}\,dx
\\
&
+\int_{\mathbb{T}^2} \Lambda^s(-\Delta\theta^1\nabla\tilde{\theta}-\Delta\tilde{\theta}\nabla\theta^2+\nabla\tilde{\theta}\phi(\theta^1))\cdot\Lambda^s\tilde{u}\,dx\\
&
+\int_{\mathbb{T}^2} \Delta\Lambda^s(\phi(\theta^1)-\phi(\theta^2)\Lambda^s\tilde{\theta}\,dx
+\int_{\mathbb{T}^2} \Lambda^s(\nabla\theta^2(\phi(\theta^1)-\phi(\theta^2)))\cdot\Lambda^s\tilde{u}\,dx=:\sum_{i=1}^7I_i.
\end{split}
\end{equation}
From Lemmas \ref{lemma-comm} and \ref{lem-4}, we have
\begin{equation*}
\begin{split}
&|I_1|=|\int_{\mathbb{T}^2}  [\nu(\theta^1)-\nu(0), \Lambda^s]D\tilde{u}:\nabla \Lambda^s\tilde{u}\,dx| \leq \|\nabla \tilde{u}\|_{H^s}\|[\nu(\theta^1)-\nu(0), \Lambda^s]D\tilde{u}\|_{L^2}^2\\
&\leq C\|\nabla \tilde{u}\|_{H^s}(\|\theta^1\|_{H^s}\|D\tilde{u}\|_{L^\infty}+\|\nabla\theta^1\|_{L^\infty}\|\tilde{u}\|_{H^{s}})\\
&\leq C(\|\theta^1\|_{H^s}\|D\tilde{u}\|_{H^1}^{\frac{s-1}{s}}\|\nabla \tilde{u}\|_{H^s}^{\frac{s+1}{s}}+\|\nabla\theta^1\|_{H^s}\|\tilde{u}\|_{H^{s}}\|\nabla \tilde{u}\|_{H^s}),
\end{split}
\end{equation*}
which implies
\begin{equation}\label{3-96}
\begin{split}
&|I_1|\leq \frac{c\nu_1}{8} \|\nabla \tilde{u}\|_{H^s}+C (\|\theta^1\|_{H^s}^{\frac{2s}{s-1}}\|\tilde{u}\|_{H^1}^2+\|\nabla\theta^1\|_{H^s}^2\|\tilde{u}\|_{H^{s}}^2).
\end{split}
\end{equation}
It is easy to find
\begin{equation}\label{3-97}
\begin{split}
&|I_2|=|-\int_{\mathbb{T}^2} \Lambda^s(\tilde{u}\cdot\nabla u^1+u^2\cdot\nabla \tilde{u})\cdot\Lambda^s\tilde{u}\,dx|\\
&\qquad\qquad\qquad\qquad\leq C(\|\tilde{u}\|_{H^s}^2\|\nabla u^1\|_{H^s}+\|\tilde{u}\|_{H^s}\|\nabla \tilde{u}\|_{H^s}\|u^2\|_{H^s})\\
&\qquad\qquad\qquad\qquad\leq \frac{c\nu_1}{8} \|\nabla \tilde{u}\|_{H^s}^2+C \|\tilde{u}\|_{H^s}^2\|u^2\|_{H^s}^2+C\|\tilde{u}\|_{H^s}^2\|\nabla u^1\|_{H^s},
\end{split}
\end{equation}
and similarly,
\begin{equation*}\label{3-99}
\begin{split}
&|I_4|+|I_5|=|\int_{\mathbb{T}^2} (\Lambda^s(\tilde{u}\cdot\nabla\theta^1) +\Lambda^s(u^2\cdot\nabla\tilde{\theta}))\Lambda^s\tilde{\theta}\,dx|\\
&\qquad\qquad\qquad +|\int_{\mathbb{T}^2} \Lambda^s(-\Delta\theta^1\nabla\tilde{\theta}-\Delta\tilde{\theta}\nabla\theta^2+\nabla\tilde{\theta}\phi(\theta^1))\cdot\Lambda^s\tilde{u}dx|\\
&\leq C (\|\tilde{u}\|_{H^s}\|\nabla\theta^1\|_{H^s} +\|u^2\|_{H^s}\|\nabla\tilde{\theta}\|_{H^s})\|\tilde{\theta}\|_{H^s}\\
&\qquad +C(\|\Delta\theta^1\|_{H^s}\|\nabla\tilde{\theta}\|_{H^s}+\|\Delta\tilde{\theta}\|_{H^s}\|\nabla\theta^2\|_{H^s}
+\|\nabla\tilde{\theta}\|_{H^s}\|\phi(\theta^1)\|_{H^s})\|\tilde{u}\|_{H^s},
\end{split}
\end{equation*}
which follows from Lemma \ref{lem-4} that
\begin{equation}\label{3-99}
\begin{split}
|I_4|+|I_5|&\leq \frac{1}{4} \|\Delta\tilde{\theta}\|_{H^s}^2+ C (\|\tilde\theta\|_{H^s}^2+\|\tilde u\|_{H^s}^2)\\
&\quad\qquad\qquad\qquad\times (\|\nabla\theta^1\|_{H^s}+\|u^2\|_{H^s}^2+\|\Delta\theta^1\|_{H^s}^2+\|\nabla\theta^2\|_{H^s}^2+\|\theta^1\|_{H^s}^2).
\end{split}
\end{equation}
Thanks to Lemma \ref{lem-4} again, we obtain
\begin{equation}\label{3-98}
\begin{split}
&|I_3|=|\int_{\mathbb{T}^2} \Lambda^s((\nu(\theta^1)-\nu(\theta^2))2Du^2):\nabla\Lambda^s\tilde{u}\,dx|\\
&\leq \frac{c\nu_1}{8}\|\nabla \tilde{u}\|_{H^s}^2+C \|(\nu(\theta^1)-\nu(\theta^2))\|_{H^s}^2\|Du^2\|_{H^s}^2\leq \frac{\nu_1}{8}\|\nabla \tilde{u}\|_{H^s}^2+C \|\tilde\theta\|_{H^s}^2\|\nabla u^2\|_{H^s}^2,
\end{split}
\end{equation}
and
\begin{equation}\label{3-100}
\begin{split}
&|I_6|+|I_7|=|\int_{\mathbb{T}^2} \Delta\Lambda^s(\phi(\theta^1)-\phi(\theta^2)\Lambda^s\tilde{\theta}\,dx|
+|\int_{\mathbb{T}^2} \Lambda^s(\nabla\theta^2(\phi(\theta^1)-\phi(\theta^2)))\cdot\Lambda^s\tilde{u}\,dx|\\
&\leq \frac{1}{4} \|\Delta\tilde{\theta}\|_{H^s}^2+ C \|\phi(\theta^1)-\phi(\theta^2)\|_{H^s}^2+C\|\nabla\theta^2\|_{H^s}\|\phi(\theta^1)-\phi(\theta^2))\|_{H^s}\|\tilde{u}\|_{H^s}\\
&\leq \frac{1}{4} \|\Delta\tilde{\theta}\|_{H^s}^2+ C (\|\tilde\theta\|_{H^s}^2+\|\tilde u\|_{H^s}^2)(1+\|\nabla\theta^2\|_{H^s}).
\end{split}
\end{equation}
Inserting (\ref{3-96})-(\ref{3-100}) into \eqref{3-95} leads to
\begin{equation}\label{3-101}
\begin{split}
&\frac{d}{dt}\|(\tilde{u},\tilde{\theta})\|_{H^s}^2+c\nu_1\|\nabla\tilde{u}\|^2_{H^s}+\|\Delta\tilde{\theta}\|^2_{H^s}
\\
&\leq C(\|\tilde\theta\|_{H^s}^2+\|\tilde u\|_{H^s}^2)(1+\|\nabla u^1\|_{H^s}+\|\nabla u^2\|_{H^s}^2+\|\nabla\theta^1\|_{H^s}+\|\Delta\theta^1\|_{H^s}^2+\|\nabla\theta^2\|_{H^s}^2).
\end{split}
\end{equation}
Thus, from \eqref{local-space}, we apply Gronwall's inequality to \eqref{3-101} to ensure
\begin{equation*}\label{3-102}
\sup_{\tau\in [0, t]}\|(\tilde{u},\tilde{\theta})(\tau)\|_{H^s}^2\leq C_T\|(\tilde{u}_0,\tilde{\theta}_0)\|^2_{H^s},
\end{equation*}
which implies continuous dependency and uniqueness of the above solution to \eqref{NSCH-2}, and then completes the proof of Theorem \ref{thm-lp-nsch}.
\end{proof}

\renewcommand{\theequation}{\thesection.\arabic{equation}}
\setcounter{equation}{0}
\section{Uniform $L^\infty$-bounds of concentrations away from singular points}\label{sect-5}
In this part, we consider the Cahn-Hilliard equation with convection
\begin{equation}\label{3-1}
\begin{cases}
&\partial_t\theta+u\cdot\nabla\theta=\Delta(\phi(\theta)-\Delta\theta) \quad\hbox{in}\quad (0,T) \times \mathbb{T}^2,\\
&\theta|_{t=0}=\theta_0
\end{cases}
\end{equation}
for given $u\in L^{\infty}([0,T); H^1(\mathbb{T}^2))\cap L^2([0,T]; H^2(\mathbb{T}^2))$ with $\partial_t u\in L^2([0,T]; L^2(\mathbb{T}^2))$ and $\dive\, u=0$.
Here $\phi=\Phi'$ satisfies \eqref{1-9}. We assume that $\theta_0 \in H^s(\mathbb{T}^2))$ with $s\geq 4$, $m(\theta_0)=0$, and $\|\theta_0\|_{L^{\infty}(\mathbb{T}^2))} \leq 1-\delta_0$ for some $\delta_0 \in (0, 1)$.

Motivated by \cite{EGZ, EMZ, MiZe2004}, our aim in this section is to prove there allow to separate the solution of \eqref{3-1} from the singular points $\pm 1$ of the free energy density $\phi$.
Estimates of this type are crucial for the study of Navier-Stokes-Cahn-Hilliard system with singular potentials, since they allow to reduce the problem to one with regular potentials, which results in the global existence of classical solution of the system \eqref{NSCH-2}. In order to achieve the goal, we consider the following approximate equation of \eqref{3-1}:
\begin{equation} \label{3-3-0}
\begin{cases}
&\partial_t\theta-\varepsilon\partial_t\Delta\theta+ u\cdot\nabla\theta=\Delta(\phi(\theta))-\Delta\theta),\quad\forall \, (t, x)\in (0,T) \times \mathbb{T}^2,
\\
&\theta|_{t=0}=\theta_0,\quad\forall \, x\in \mathbb{T}^2
\end{cases}
\end{equation}
with $\varepsilon>0$, which may be equivalently rewritten as
\begin{equation} \label{3-3}
\begin{cases}
&\varepsilon\partial_t\theta+(-\Delta)^{-1}((u\cdot\nabla)\theta+\partial_t\theta)=\Delta\theta-\phi(\theta)+m(\phi(\theta)),\quad\forall \, (t, x)\in (0,T) \times \mathbb{T}^2,
\\
&\theta|_{t=0}=\theta_0,\quad\forall \, x\in \mathbb{T}^2.
\end{cases}
\end{equation}
We define the phase space $\mathbb{D}_{0, \varepsilon}$ for the approximate equation \eqref{3-3} as
\begin{equation}\label{3-4}
\begin{split}
&\mathbb{D}_{0, \varepsilon}:=\{f\in H^2(\mathbb{T}^2)| \|f\|_{L^\infty(\mathbb{T}^2)}\leq 1, m(f_0)=0, \phi(f)\in L^2(\mathbb{T}^2), \varepsilon^{\frac{1}{2}}\varphi\in L^2(\mathbb{T}^2),
\\
&\quad \qquad\,\varphi\in \dot{H}^{-1}(\mathbb{T}^2), \,\hbox{where}\, \varphi=\varphi(f):=(\varepsilon+(-\Delta)^{-1})^{-1}[\Delta f-\phi(f)+m(\phi(f))]\}
\end{split}
\end{equation}
 equipped with the norm
\begin{equation}\label{3-4-1}
\|f\|^2_{\mathbb{D}_{0, \varepsilon}}:=\|f\|^2_{H^2(\mathbb{T}^2)}+\|\phi(f)\|^2_{L^2(\mathbb{T}^2)}
+\varepsilon\|\varphi(f)\|^2_{L^2(\mathbb{T}^2)}+\|\varphi(f)\|^2_{\dot{H}^{-1}(\mathbb{T}^2)},
\end{equation}
and also the phase space $\mathbb{D}_{0}$ for the original equation \eqref{3-1} as
\begin{equation}\label{3-4-2}
\begin{split}
&\mathbb{D}_{0}:=\{f\in H^2(\mathbb{T}^2)| \|f\|_{L^\infty(\mathbb{T}^2)}\leq 1, m(f_0)=0, \\
&\qquad\qquad\qquad\qquad\qquad\qquad \phi(f)\in L^2(\mathbb{T}^2),\,\nabla(\Delta f-\phi(f))\in L^2(\mathbb{T}^2)\}
\end{split}
\end{equation}
 equipped with the norm
\begin{equation}\label{3-4-3}
\|f\|^2_{\mathbb{D}_{0}}:=\|f\|^2_{H^2(\mathbb{T}^2)}+\|\phi(f)\|^2_{L^2(\mathbb{T}^2)}
+\|\nabla(\Delta f-\phi(f))\|^2_{L^2(\mathbb{T}^2)},
\end{equation}
Repeating the argument of the proof of Theorem \ref{thm-lp-nsch}, we may readily find the unique smooth solution $\theta$ to \eqref{3-3} on $[0, T_2)$ for $0<T_2 \leq T$, and moreover, if the maximal existence time $T_2<T$ and
\begin{equation}\label{3-5-1}
\|\theta\|_{L^{\infty}((0, T_2) \times \mathbb{T}^2)}\leq 1-\delta_1
\end{equation}
for some positive constant $\delta_1 (\in (0, 1))$ independent of $\varepsilon$, then $\theta$
may be continued beyond $T_2$ into a smooth solution of \eqref{3-3}.

In what follows, we only give a derivation of the a priori estimates for the solution $\theta$ to \eqref{3-3}, assuming that $\theta$ is a sufficiently regular function, and then deduce \eqref{3-5-1}, which will allow us to extend the existence time interval $[0, T_2)$ to $[0, T)$.

For this, we need first to establish the following dissipative estimate in the space $\mathbb{D}_{0, \varepsilon}$.

\begin{lem}\label{thm-uni-ch}
Let the non-linearity $\phi=\Phi'$ satisfy \eqref{1-9}, $\varepsilon \in (0, 1]$, $\theta_0 \in \mathbb{D}_{0, \varepsilon}$, $\|\theta_0\|_{L^\infty}\leq 1-\delta_0$ for some $\delta_0 \in(0,1)$, $u\in L^{\infty}([0,T]; H^1(\mathbb{T}^2))\cap L^2([0,T]; H^2(\mathbb{T}^2))$ with $\partial_t u\in L^2([0,T]; L^2(\mathbb{T}^2))$ and $\dive u=0$, and $\theta(t)$ be a weak solution to\eqref{3-3} on $[0,T)$ for $0 <T <+\infty$ satisfying \eqref{3-56-2},
$\theta \in C([0,T); \mathbb{D}_{0, \varepsilon}))$, $\partial_t \theta \in L^2_{loc}([0,T); H^{1}(\mathbb{T}^2))$. Then there holds
\begin{equation}\label{3-6}
\begin{split}
\sup_{\tau \in [0, t]}\|\theta(\tau)\|^2_{\mathbb{D}_{0, \varepsilon}}+\int^t_0\|\partial_t\theta(\tau)\|^2_{H^1(\mathbb{T}^2)}\,d\tau\leq C(\|\theta_0\|_{\mathbb{D}_{0, \varepsilon}}, T).
\end{split}
\end{equation}
\end{lem}

\begin{proof}
Setting $\varphi(t):=\partial_t\theta(t)$. Thanks to the fact that $\theta_0 \in \mathbb{D}_{0, \varepsilon}$, we may first verify that $m(\theta(t))\equiv 0$ and then $m(\varphi(t))\equiv 0$ on the existence time interval according the equation \eqref{3-3}.  Differentiate \eqref{3-3} with respect to $t$, and then multiply the resulting equation by $\varphi(t)$ and integrate over $\mathbb{T}^2$,  we get
\begin{equation}\label{3-7}
\begin{split}
&\frac{1}{2}\frac{d}{dt}(\varepsilon\|\varphi(t)\|^2_{L^2}+\|\varphi(t)\|^2_{\dot{H}^{-1}}) +\|\varphi\|^2_{H^1}
\\
&=-\int_{\mathbb{T}^2} (-\Delta)^{-1}\partial_t((u\cdot\nabla)\theta)\varphi\, dx+\|\varphi\|^2_{L^2}-\int_{\mathbb{T}^2} \phi'(\theta)|\varphi|^2\,dx+ m(\phi'(\theta)\varphi)\int_{\mathbb{T}^2} \varphi \,dx.
\end{split}
\end{equation}
For the first integral of the right-hand side in \eqref{3-7}, one can obtain
\begin{equation}\label{3-9}
\begin{split}
&|\int_{\mathbb{T}^2} (-\Delta)^{-1}\partial_t((u\cdot\nabla)\theta)\varphi \,dx|=|\int_{\mathbb{T}^2} (\partial_tu\theta+u\partial_t\theta)\cdot (-\Delta)^{-1}\nabla\varphi \,dx|\\
&\leq\frac{1}{2}\|\partial_tu\theta+u\partial_t\theta\|_{L^2}^2+\frac{1}{2}\|\varphi\|^2_{\dot{H}^{-1}}
\leq C(\|u_t\|^2_{L^2}\|\theta\|^2_{L^\infty}+\|u\|^2_{L^\infty}\|\varphi\|^2_{L^2})+\frac{1}{2}\|\varphi\|^2_{\dot{H}^{-1}}
\\
&\leq \frac{1}{8}\|\varphi\|^2_{H^1}+C (\|u\|^4_{L^\infty}+1)\|\varphi\|^2_{\dot{H}^{-1}}+C\|u_t\|^2_{L^2}
\end{split}
\end{equation}
due to \eqref{3-56-2} and the interpolation inequality. On the other hand, applying the interpolation inequality $\|\varphi\|^2_{L^2(\mathbb{T}^2)}\leq C\|\varphi\|_{H^1(\mathbb{T}^2)}\|\varphi\|_{\dot{H}^{-1}(\mathbb{T}^2)}$
and noting that $m(\varphi(t))\equiv0$ and $\phi'(\theta)\geq-\alpha$, we deduce
\begin{equation}\label{3-8}
\|\varphi\|^2_{L^2}-\int_{\mathbb{T}^2} \phi'(\theta)|\varphi|^2dx+ m(\phi'(\theta)\varphi)\int_{\mathbb{T}^2} \varphi dx\leq \frac{1}{8}\|\varphi\|^2_{H^1}+C \|\varphi\|^2_{\dot{H}^{-1}}.
\end{equation}
Plugging \eqref{3-7} and \eqref{3-8} into \eqref{3-9} results in
\begin{equation}\label{3-10}
\begin{split}
&\frac{d}{dt}(\varepsilon\|\varphi(t)\|^2_{L^2}+\|\varphi(t)\|^2_{\dot{H}^{-1}}) +\|\varphi\|^2_{H^1}
\leq C\|u\|^4_{L^\infty}\|\varphi\|^2_{\dot{H}^{-1}}+C\|u_t\|^2_{L^2}+C_\alpha\|\varphi\|^2_{\dot{H}^{-1}}\\
&\leq C\| u\|^2_{L^2}\|\Delta u\|^2_{L^2}\|\varphi\|^2_{\dot{H}^{-1}}+C\|u_t\|^2_{L^2}+C \|\varphi\|^2_{\dot{H}^{-1}},
\end{split}
\end{equation}
where we have used the interpolation inequality $\|u\|^4_{L^\infty}\leq C\| u\|^2_{L^2}\|\Delta u\|^2_{L^2}$ in the second inequality.

Hence, applying Gronwall's inequality to \eqref{3-10} results in
\begin{equation}\label{3-11}
\begin{split}
&\sup_{\tau \in [0, t]}(\varepsilon\|\varphi(\tau)\|^2_{L^2}+\|\varphi(\tau)\|^2_{\dot{H}^{-1}})+\int^t_0\|\varphi(\tau)\|^2_{H^1}\,d\tau\\
&\leq 2(\varepsilon\|\varphi(0)\|^2_{L^2}+\|\varphi(0)\|^2_{\dot{H}^{-1}}+\|u_t\|_{L^2([0,t];L^2)}^2)e^{C  t+C\| u\|^2_{L^{\infty}([0,t];L^2)}\|\Delta u\|^2_{L^2([0,t];L^2)}}
\\
&\leq C_{T}(\varepsilon\|\varphi(0)\|^2_{L^2}+\|\varphi(0)\|^2_{\dot{H}^{-1}}+1),
\end{split}
\end{equation}
which implies the $\varphi$-part of \eqref{3-6}.

For the estimate of the $\theta$-part of \eqref{3-6}, similar to the proof of \eqref{3-60}, we may get
\begin{equation*}
\frac{d}{dt}(\|\theta\|^2_{L^2}+\varepsilon\|\nabla\theta\|^2_{L^2})+\|\theta\|^2_{H^2} \leq C\|\theta\|^2_{L^2},
\end{equation*}
which along with Gronwall's inequality gives rise to
\begin{equation}\label{3-11-2}
\sup_{\tau\in [0, t]}(\|\theta(\tau)\|^2_{L^2}+\varepsilon\|\nabla\theta(\tau)\|^2_{L^2})+\int^t_0\|\theta\|^2_{H^2}\,d\tau \leq C e^{Ct}\|\theta_0\|_{L^2}^2.
\end{equation}
On the other hand, we rewrite \eqref{3-3} as
\begin{equation}\label{3-12}
\Delta\theta-\phi(\theta)+m(\phi(\theta))=h(t):=\varepsilon\varphi(t)+(-\Delta)^{-1}\varphi(t)+(-\Delta)^{-1}((u\cdot\nabla)\theta)(t).
\end{equation}
Multiply \eqref{3-12} by $\Delta\theta$ and integrate over $\mathbb{T}^2$, and note that $m(\Delta\theta)=0$ and $\phi'\geq-\alpha$, then we find
\begin{equation}\label{3-13}
\begin{split}
&\|\Delta\theta\|^2_{L^2}=-\int_{\mathbb{T}^2} \phi'(\theta)|\nabla\theta|^2dx+m(\phi(\theta))\int_{\mathbb{T}^2} \Delta\theta \,dx
\\
&+\int_{\mathbb{T}^2} (\varepsilon\varphi(t)+(-\Delta)^{-1}\varphi(t))\Delta\theta dx+\int_{\mathbb{T}^2} (-\Delta)^{-1}((u\cdot\nabla)\theta)(t)\Delta\theta \,dx
\\
& \leq\alpha\|\nabla\theta\|^2_{L^2}+\int_{\mathbb{T}^2} (\varepsilon\varphi(t)+(-\Delta)^{-1}\varphi(t))\Delta\theta \, dx+\int_{\mathbb{T}^2} (-\Delta)^{-1}((u\cdot\nabla)\theta)(t)\Delta\theta\, dx.
\end{split}
\end{equation}
It is easy to find
\begin{equation}\label{3-15}
\int_{\mathbb{T}^2} (-\Delta)^{-1}((u\cdot\nabla)\theta)\,\Delta\theta \,dx=\int_{\mathbb{T}^2} (u\cdot\nabla)\theta\theta \,dx=0,
\end{equation}
and
\begin{equation*}
\begin{split}
&\int_{\mathbb{T}^2} (\varepsilon\varphi(t)+(-\Delta)^{-1}\varphi(t))\Delta\theta \,dx\leq \frac{1}{8}\|\Delta\theta\|^2_{L^2}+\varepsilon^2 \|\varphi\|^2_{L^2}+\|\varphi\|_{\dot{H}^{-1}}\|\nabla\theta\|_{L^2},
\end{split}
\end{equation*}
which along with $\alpha\|\nabla\theta\|^2_{L^2}\leq\frac{1}{8}\|\Delta\theta\|^2_{L^2}+C\|\theta\|^2_{L^2}$
implies
\begin{equation}\label{3-16}
\begin{split}
&\alpha\|\nabla\theta\|^2_{L^2}+\int_{\mathbb{T}^2} (\varepsilon\varphi(t)+(-\Delta)^{-1}\varphi(t))\Delta\theta \,dx\\
&\qquad \leq \frac{1}{2}\|\Delta\theta\|^2_{L^2}+\varepsilon^2 C \|\varphi\|^2_{L^2}+C\|\varphi\|^2_{\dot{H}^{-1}}+C\|\theta\|^2_{L^2}.
\end{split}
\end{equation}
Inserting \eqref{3-15} and \eqref{3-16} into \eqref{3-13} yields
\begin{equation}\label{3-17}
\sup_{\tau \in [0, t]}\|\theta(\tau)\|^2_{H^2}\leq C\sup_{\tau \in [0, t]}\bigg(\varepsilon^2  \|\varphi\|^2_{L^2}+\|\varphi\|^2_{\dot{H}^{-1}}+ \|\theta\|^2_{L^2}\bigg)\leq C(\|\theta_0\|_{\mathbb{D}_{0, \varepsilon}}, T),
\end{equation}
where we have used the inequalities \eqref{3-11} and \eqref{3-11-2}.

Let's now estimate $\|\phi(\theta)\|^2_{L^2}$. It follows from \eqref{3-11}, \eqref{3-12}, \eqref{3-17}, $\dive\, u=0$, and $m(\phi)\equiv 0$ that
\begin{equation}\label{3-18-1}
\begin{split}
&\|\phi(\theta)-m(\phi(\theta))\|_{L^2(\mathbb{T}^2)}\leq \|\Delta\theta\|_{L^2}+ \|h\|_{L^2(\mathbb{T}^2)}\\
&\leq \|\Delta\theta\|_{L^2}+\|\varepsilon\varphi\|_{L^2}+\|(-\Delta)^{-1}\varphi\|_{L^2}+\|(-\Delta)^{-1}\dive\,(u\,\theta)\|_{L^2}\\
  & \leq \|\Delta\theta\|_{L^2}+\varepsilon\|\varphi\|_{L^2}+C\|\varphi\|_{\dot H^{-1}}+C\|u\|_{L^2}\|\theta\|_{H^2},
\end{split}
\end{equation}
which, together with \eqref{3-11}, \eqref{3-11-2}, and  \eqref{3-17}, implies
\begin{equation}\label{3-18}
\begin{split}
&\sup_{\tau \in [0, t]}\|\phi(\theta)-m(\phi(\theta))\|_{L^2(\mathbb{T}^2)}\leq C(\|\theta_0\|_{\mathbb{D}_{0, \varepsilon}}, T).
\end{split}
\end{equation}
Whence applying Proposition \ref{prop-2} to \eqref{3-12} (where we take $f:=\Delta\theta-h$), we deduce that
\begin{equation}\label{3-19}
\begin{split}
|m(\phi(\theta))|\leq C (1+|\mathbb{T}|\|f\|_{L^2})\leq C (1+\|\phi(\theta)-m(\phi(\theta))\|_{L^2(\mathbb{T}^2)}),
\end{split}
\end{equation}
and then
\begin{equation}\label{3-19-1}
\begin{split}
\sup_{\tau \in [0, t]}|m(\phi(\theta))|\leq C(\|\theta_0\|_{\mathbb{D}_{0, \varepsilon}}, T).
\end{split}
\end{equation}
Note that
\begin{equation}\label{3-20}
\begin{split}
&\|\phi(\theta)-m(\phi(\theta))\|^2_{L^2}=\int_{\mathbb{T}^2} (|\phi(\theta)|^2-2\phi(\theta)m(\phi(\theta))+|m(\phi(\theta))|^2)dx
\\
&=\|\phi(\theta)\|^2_{L^2}-|\mathbb{T}^2| |m(\phi(\theta))|^2,
\end{split}
\end{equation}
it follows that
\begin{equation*}
\|\phi(\theta)\|^2_{L^2}= \|\phi(\theta)-m(\phi(\theta))\|^2_{L^2}+|\mathbb{T}^2| |m(\phi(\theta))|^2,
\end{equation*}
which, together with (\ref{3-17}-\ref{3-20}), gives rise to
\begin{equation}\label{3-21}
\sup_{\tau \in [0, t]}\|\phi(\theta)\|^2_{L^2}\leq C(\|\theta_0\|_{\mathbb{D}_{0, \varepsilon}}, T).
\end{equation}
Estimates (\ref{3-11}), (\ref{3-17}) and (\ref{3-21}) finish the proof of Lemma \ref{thm-uni-ch}.
\end{proof}
\begin{rmk}\label{rmk-indep-1}
Let's point out that the constant $C(\|\theta_0\|_{\mathbb{D}_{0, \varepsilon}}, T)$ in \eqref{3-6} is independent of $\varepsilon \in (0, 1)$, which can be readily verified by using the definition of $\|\theta_0\|_{\mathbb{D}_{0, \varepsilon}}$.
\end{rmk}

The next lemma gives an estimate of $\phi(\theta)$ in the space $L^2([0,t]; L^\infty(\mathbb{T}^2))$, which implies that $\theta$ doesn't touch the singular points for almost every time.

\begin{lem}\label{col-uni-ch}
Under the assumptions in Theorem \ref{thm-uni-ch}, the solution $\theta$ of problem (\ref{3-3}) satisfies the following estimates:
\begin{equation}\label{3-22}
\int^T_0\|\phi(\theta(\tau))\|^2_{L^\infty(\mathbb{T}^2)}\, d\tau\leq C_{T} \quad \mbox{and}
 \quad \|\theta\|_{L^{\infty}((0, T)\times \mathbb{T}^2)} <1,
 \end{equation}
where the constant $C_{T}$ depends on $T$, but not $\varepsilon \in (0, 1)$.
\end{lem}

\begin{proof}
Firstly, we rewrite the problem \eqref{3-3} as a second order parabolic equation:
\begin{equation}\label{3-23}
\varepsilon\partial_t\theta-\Delta\theta+\phi(\theta)=\tilde{h}:=m(\phi(\theta))-(-\Delta)^{-1}\partial_t\theta-(-\Delta)^{-1}((u\cdot\nabla)\theta).
\end{equation}
According to \eqref{3-6} and the interpolation inequality $\|f\|_{L^\infty(\mathbb{T}^2)}\leq C \|f\|_{L^2(\mathbb{T}^2)}^{\frac{1}{2}}\|\Delta f\|_{L^2(\mathbb{T}^2)}^{\frac{1}{2}}$, we have
\begin{equation}\label{3-24}
\begin{split}
\|(-\Delta)^{-1}\partial_t\theta\|^2_{L^2([0,T];L^\infty(\mathbb{T}^2))}&\leq C\int_0^T\|(-\Delta)^{-1}\partial_t\theta\|_{L^2}\|\partial_t\theta\|_{L^2}\, d\tau\\
&\leq C \|\partial_t\theta\|^2_{L^2([0,T]; L^2(\mathbb{T}^2))}\leq C_T.
\end{split}
\end{equation}
Note that
\begin{equation}\label{3-25}
\begin{split}
&\|(-\Delta)^{-1}((u\cdot\nabla)\theta)\|^2_{L^2([0,t];L^\infty(\mathbb{T}^2))}
\leq C\|(-\Delta)^{-1}\dive\,(u\theta)\|_{L^2([0,t];L^\infty)}^2
\\
&\leq C\|u\theta\|_{L^2([0,t]; L^4)}^2 \leq C\|u\|^2_{L^2([0,t]; L^{4})}\|\theta\|^2_{L^\infty([0,t] \times \mathbb{T}^2)}\leq C_T,
\end{split}
\end{equation}
so, we get from \eqref{3-19-1} and \eqref{3-23}-\eqref{3-25} that
\begin{equation}\label{3-26}
\|\tilde{h}\|^2_{L^2([0,t];L^\infty(\mathbb{T}^2))}\leq C_T, \quad\forall~ 0\leq t\leq T.
\end{equation}

We set $h_{\pm}(t):=\|\tilde{h}(t)\|_{L^\infty(\mathbb{T}^2)}$ and consider the following two auxiliary ODEs:
\begin{equation*}
\begin{cases}
&\varepsilon y'_{\pm}+\phi(y_{\pm})=h_{\pm}, \quad \forall~t\geq0,
\\
&y_{\pm}(0)=\pm\|\theta_0\|_{L^\infty(\mathbb{T}^2)}.
\end{cases}
\end{equation*}
Then, thanks to Lemma \ref{thm-uni-ch}, the solutions $y_{\pm}(t)$ are well defined.

Define the operator $\mathcal{L}$ as
\begin{equation*}
\mathcal{L}\theta(t):=\varepsilon\partial_t\theta-\Delta\theta+\phi(\theta)-\tilde{h}.
\end{equation*}
Note that
\begin{equation*}
\begin{split}
&\mathcal{L}\theta(t)=0,\quad \mathcal{L}y_+(t)=\varepsilon y'_+-\Delta y_++\phi(y_+)-\tilde{h}=h_+-\tilde{h}\geq0,
\\
&\mathcal{L}y_-(t)=\varepsilon y'_--\Delta y_-+\phi(y_-)-\tilde{h}=h_--\tilde{h}\leq0,
\end{split}
\end{equation*}
then we apply the comparison principle of the second-order parabolic equation to get
\begin{equation}\label{3-31}
y_-(t)\leq\theta(t,x)\leq y_+(t), \quad\forall~t\geq0,\,~x\in\mathbb{T}^2.
\end{equation}

On the other hand, it follows from Proposition \ref{prop-4} and (\ref{3-26}) that
\begin{equation}\label{3-32}
\int^t_0|\phi(y_{\pm}(\tau))|^2\,d\tau\leq C_T (1+\|h_\pm\|^2_{L^2([0,T])})\leq C_T.
\end{equation}
Setting $\widetilde{\phi}(z):=\phi(z)-\phi(0)+\alpha\, z$ for $z \in (-1, 1)$, we have $\widetilde{\phi}'(z)>0 $ for any $z \in (-1, 1)$, which implies that $\widetilde{\phi}(\cdot)$ is increasing in $(-1, 1)$.
Therefore, from \eqref{3-26}, \eqref{3-31}, and \eqref{3-32}, we deduce
\begin{equation*}\label{3-33}
\begin{split}
&\int^t_0\|\phi(\theta(\tau))\|^2_{L^\infty}\, d\tau\leq \int^t_0\|\widetilde{\phi}(\theta(\tau))\|^2_{L^\infty}\, d\tau+C\int^t_0(1+\|\theta(\tau)\|^2_{L^\infty})\, d\tau\\
& \leq \int^t_0|\widetilde{\phi}(y_{+}(\tau))|^2\,d\tau+\int^t_0|\widetilde{\phi}(y_{-}(\tau))|^2\,d\tau+C_T\\
&\leq \int^t_0(|{\phi}(y_{+}(\tau))|^2+|{\phi}(y_{-}(\tau))|^2+C|y_{+}(\tau)|^2+C|y_{-}(\tau)|^2+C)\,d\tau+C_T\leq C_{T},
\end{split}
\end{equation*}
which leads to $\|\theta(t)\|_{L^{\infty}(\mathbb{T}^2)} <1$ for almost every $ t \in (0, T)$ according to Assumption \ref{assumption-1}.
This ends the proof of Lemma \ref{col-uni-ch}.
\end{proof}

In order to separate $\theta$ from the singular points of $\phi$, we need to investigate the integrability of $|\phi'(\theta)|$, which is stated as follows.

\begin{lem}\label{lem-Uni-1}
We assume that the non-linearity $\phi=\Phi'$ satisfies (\ref{1-9}). Then, for $\forall \, p\in[1,+\infty)$, the following estimate holds:
\begin{equation*}
\int^t_0\int_{\mathbb{T}^2}|\phi'(\theta)|^p dxds\leq C_{p, T}, \quad \forall \, t \in (0, T),
\end{equation*}
where the constant $C_{p,T}$ is independent of $\varepsilon \in (0, 1)$.
\end{lem}

\begin{proof}
As in the proof of Lemma \ref{col-uni-ch}, we first denote $\widetilde{\phi}(z):=\phi(z)-\phi(0)+\alpha\, z$ for $z \in (-1, 1)$, which yields $\widetilde{\phi}'(z)\geq 0 $ for any $z \in (-1, 1)$. Rewrite \eqref{3-3} as
\begin{equation}\label{3-35}
\varepsilon\partial_t\theta-\Delta\theta+\widetilde{\phi}(\theta)
=\bar{h}:=m(\phi(\theta))-(-\Delta)^{-1}\partial_t\theta-(-\Delta)^{-1}((u\cdot\nabla)\theta)+\alpha\theta-\phi(0).
\end{equation}
where the function $\bar{h}$ satisfies
\begin{equation}\label{3-36}
\|\bar{h}\|_{L^\infty([0, T]; H^1(\mathbb{T}^2))}\leq C_{T}
\end{equation}
due to the estimate \eqref{3-6}.

Let $L>0$ be an arbitrary positive number, and $\Phi_L(z):=\int^z_0\widetilde{\phi}(\eta)e^{L|\widetilde{\phi}(\eta)|}d\eta\geq 0$ for any $z\in (-1, 1)$. Multiply \eqref{3-35} by $\widetilde{\phi}(\theta)e^{L|\widetilde{\phi}(\theta)|}$, and then integrate over
$(0,t)\times\mathbb{T}^2$, we have
\begin{equation}\label{3-38}
\begin{split}
&\varepsilon\int_{\mathbb{T}^2} \Phi_L(\theta(t))dx-\int^t_0\int_{\mathbb{T}^2} \Delta\theta\widetilde{\phi}(\theta(\tau))e^{L|\widetilde{\phi}(\theta(\tau))|}\,dxd\tau
+\int^t_0\int_{\mathbb{T}^2} |\widetilde{\phi}(\theta(\tau))|^2e^{L|\widetilde{\phi}(\theta(\tau))|}\,dxd\tau
\\
=&\int^t_0\int_{\mathbb{T}^2} \bar{h}\widetilde{\phi}(\theta(\tau))e^{L|\widetilde{\phi}(\theta(\tau))|}\,dxd\tau+\varepsilon\int_{\mathbb{T}^2} \Phi_L(\theta(0))\,dx,
\end{split}
\end{equation}
Since $\widetilde{\phi}'(z)>0 $ for any $z \in (-1, 1)$, one can get
\begin{equation*}\label{3-39}
-\int^t_0\int_{\mathbb{T}^2} \Delta\theta\widetilde{\phi}(\theta)e^{L|\widetilde{\phi}(\theta)|}\,dx
=\int^t_0\int_{\mathbb{T}^2} |\nabla\theta|^2\widetilde{\phi}'(\theta)[1+L|\widetilde{\phi}(\theta)|]e^{L|\tilde{\phi}(\theta)|}\,dxd\tau\geq 0,
\end{equation*}
which follows from \eqref{3-38} that
\begin{equation}\label{3-40}
\int^t_0\int_{\mathbb{T}^2} |\widetilde{\phi}(\theta(\tau))|^2e^{L|\widetilde{\phi}(\theta(\tau))|}\,dxd\tau
\leq\int^t_0\int_{\mathbb{T}^2} |\bar{h}(\tau)||\widetilde{\phi}(\theta(\tau))|e^{L|\widetilde{\phi}(\theta(\tau))|}\,dxd\tau+C_{T, L},
\end{equation}
where the constant $C_{T, L}$ is independent of $\varepsilon \in (0, 1)$. Applying Proposition \ref{prop-O-Y} to the integrand in the right-hand side of \eqref{3-40}, and taking $p=N|\bar{h}(\tau)|$ and $q=N^{-1}|\widetilde{\phi}(\theta(\tau))|e^{L|\widetilde{\phi}(\theta(\tau))|}$ in \eqref{2-4}, where $N \geq 1$ be an arbitrary positive number, we obtain from \eqref{2-5} and \eqref{2-6} that
\begin{equation}\label{3-41}
\begin{split}
|\bar{h}||\widetilde{\phi}(\theta)|e^{L|\widetilde{\phi}(\theta)|}=p\cdot q\leq A(p)+\tilde{A}(q),
\end{split}
\end{equation}
where
\begin{equation}\label{3-41-1}
\begin{split}
&A(p):=e^p-p-1 \leq e^{p}=e^{N|\bar{h}(\tau)|},\\
&\widetilde{A}(q):=(1+q)\ln(1+q)-q\leq q\ln(1+q)\\
&\leq N^{-1}|\widetilde{\phi}(\theta)|e^{L|\widetilde{\phi}(\theta)|}\ln(e^{2L|\widetilde{\phi}(\theta)|})
=\frac{2L}{N}|\widetilde{\phi}(\theta)|^2e^{L|\widetilde{\phi}(\theta)|}.
\end{split}
\end{equation}
Therefore, taking $N=N(L)$ sufficiently large in \eqref{3-41-1}, we obtain from \eqref{3-41} that
\begin{equation}
|\bar{h}||\widetilde{\phi}(\theta)|e^{L|\widetilde{\phi}(\theta)|}\leq\frac{1}{2}|\widetilde{\phi}(\theta)|^2e^{L|\widetilde{\phi}(\theta)|}+e^{N|\bar{h}|}.\label{3-42}
\end{equation}
Inserting \eqref{3-42} into the right-hand side of \eqref{3-40} yields
\begin{equation} \label{3-43}
\begin{split}
&\int^t_0\int_{\mathbb{T}^2} |\widetilde{\phi}(\theta)|^2e^{L|\widetilde{\phi}(\theta)|}\,dxd\tau\leq\int^t_0\int_{\mathbb{T}^2}  e^{N|\bar{h}|}\,dxd\tau+C_{T, L}.
\end{split}
\end{equation}
Using (\ref{3-36}), (\ref{3-43}) and Orlicz embedding theorem (Lemma \ref{thm-O-Em}), we obtain
\begin{equation}\label{3-44}
\begin{split}
&\int^t_0\int_{\mathbb{T}^2} |\widetilde{\phi}(\theta)|^2e^{L|\widetilde{\phi}(\theta)|}\,dxd\tau\leq C_{T, L}\int^t_0(e^{C(N)\|\bar{h}\|^2_{H^1}}\,d\tau+1)\leq C_{T, L}.
\end{split}
\end{equation}
Therefore, we deduce from \eqref{3-44} that
\begin{equation}\label{3-45}
\begin{split}
&\int_0^t\int_{\mathbb{T}^2}  e^{L|\phi(\theta(\tau))|}\,dxd\tau \leq\int_0^t\int_{\mathbb{T}^2}  e^{L|\widetilde{\phi}(\theta(\tau))|+C L}\,dxd\tau
\\
&\leq C\int_0^t\int_{\Omega_0}e^{(1+C)L}\, dxd\tau +C\int_0^t\int_{\mathbb{T}^2\setminus\Omega_0}|\widetilde{\phi}(\theta(\tau))|^2e^{L|\tilde{\phi}(\theta(\tau))|}\,dxd\tau \leq C_{T, L},
\end{split}
\end{equation}
where $\Omega_0:=\{x\in\mathbb{T}^2|\,\widetilde{\phi}(\theta(\tau, x))\leq1\}$.

It follows from \eqref{1-9} and \eqref{3-45} that $\forall~p\in[1,+\infty)$
\begin{equation*}\label{3-46}
\begin{split}
&\int^t_0\int_{\mathbb{T}^2}|\phi'(\theta)|^p dxd\tau \leq\int^t_0\int_{\mathbb{T}^2}\bigg(C_1e^{C_2|\phi(\theta)|}+C_3\bigg)^p \,dxd\tau
\\
&\leq C_{p}\int^t_0\int_{\mathbb{T}^2}(e^{pC_2|\phi(\theta)|}+1)\,dxds\leq C_{T, p},
\end{split}
\end{equation*}
where the constant $C_{T, p}$ is independent of $\varepsilon \in (0, 1]$. This completes the proof of Lemma \ref{lem-Uni-1}.
\end{proof}
With Lemma \ref{lem-Uni-1} in hand, we are now in a position to investigate the uniform $L^\infty$-bounds of the approximate solution away from singular points.
\begin{thm}\label{thm-Uni-M}
Under the assumptions in Lemma \ref{thm-uni-ch}, if, in addition, $\theta_0\in H^s$ with $s\geq 4$, then, there exists a positive constant $\delta=\delta(\delta_0, T)$ such that
\begin{equation}\label{3-47}
\|\theta\|_{L^\infty([0, T)\times \mathbb{T}^2)}\leq 1-\delta.
\end{equation}
\end{thm}

\begin{proof}
We differentiate \eqref{3-3} with respect to $t$ and set $\varphi:=\partial_t\theta$. Then we have
\begin{equation}\label{3-48}
\begin{cases}
&\varepsilon\partial_t\varphi+(-\Delta)^{-1}\partial_t((u\cdot\nabla)\theta+\varphi)-\Delta\varphi=m(\phi'(\theta)\varphi)-\phi'(\theta)\varphi, \quad \mbox{in} \quad (0,T) \times \mathbb{T}^2,
\\
&\varphi|_{t=0}=\varphi_0.
\end{cases}
\end{equation}
Multiplying $-\Delta\varphi$ on \eqref{3-48} and integrating over $\mathbb{T}^2$, we derive from $m(\Delta\varphi)=0$ that
\begin{equation}\label{3-49}
\begin{split}
&\frac{1}{2}\frac{d}{dt}(\varepsilon\|\nabla\varphi\|^2_{L^2}+\|\varphi\|^2_{L^2})+\|\Delta\varphi\|^2_{L^2}
\\
&=\int_{\mathbb{T}^2} ((u_t\cdot\nabla)\theta+u\cdot\nabla\varphi)\varphi\, dx+\int_{\mathbb{T}^2} \phi'(\theta)\varphi\Delta\varphi dx
\end{split}
\end{equation}
Thanks to $\dive\, u=0$, we get
\begin{equation}\label{3-50}
\begin{split}
&\int_{\mathbb{T}^2} ((u_t\cdot\nabla)\theta+u\cdot\nabla\varphi)\varphi \,dx=-\int_{\mathbb{T}^2}  \theta\, u_t\cdot \nabla \varphi \,dx\leq C\|u_t\|_{L^2}\|\theta\|_{L^\infty}\|\nabla\varphi\|_{L^2}
\\
&\leq C\|u_t\|_{L^2}^2+C\|\nabla\varphi\|^2_{L^2}\leq \frac{1}{8}\|\Delta\varphi\|_{L^2}^2 +C(\|u_t\|_{L^2}^2+\|\varphi\|^2_{L^2}).
\end{split}
\end{equation}
Note that
\begin{equation*}\label{3-51}
\begin{split}
&\int_{\mathbb{T}^2} \phi'(\theta)\varphi\Delta\varphi \,dx=\int_{\mathbb{T}^2} \phi'(\theta)\varphi\Delta\varphi \,dx\leq \frac{1}{8}\|\Delta\varphi\|_{L^2}^2+C \int_{\mathbb{T}^2} |\phi'(\theta)\varphi|^2\,dx,
\end{split}
\end{equation*}
which, together with
\begin{equation*}
\begin{split}
&C\int_{\mathbb{T}^2} |\phi'(\theta)\varphi|^2\,dx\leq C\|\phi'(\theta)\|^2_{L^4}\|\varphi\|^2_{L^4}\leq C\|\phi'(\theta)\|^2_{L^4}\|\varphi\|_{L^2}\|\nabla\varphi\|_{L^2}
\\
&\leq C\|\phi'(\theta)\|^4_{L^4}\|\varphi\|^2_{L^2}+\|\nabla\varphi\|^2_{L^2}\leq \frac{1}{8}\|\Delta\varphi\|^2_{L^2}+C (\|\phi'(\theta)\|^4_{L^4}+1)\|\varphi\|^2_{L^2},
\end{split}
\end{equation*}
follows that
\begin{equation}\label{3-52}
\begin{split}
&\int_{\mathbb{T}^2} \phi'(\theta)\varphi\Delta\varphi \,dx\leq \frac{1}{4}\|\Delta\varphi\|_{L^2}^2+C (\|\phi'(\theta)\|^4_{L^4}+1)\|\varphi\|^2_{L^2}.
\end{split}
\end{equation}
Substituting \eqref{3-50} and \eqref{3-52} into \eqref{3-49} yields
\begin{equation*}\label{3-53}
\frac{d}{dt}(\varepsilon\|\nabla\varphi\|^2_{L^2}+\|\varphi\|^2_{L^2})+\|\Delta\varphi\|^2_{L^2}\leq C (\|\phi'(\theta)\|^4_{L^4}+1)\|\varphi\|^2_{L^2}+C \|u_t\|^2_{L^2}.
\end{equation*}
Hence, thanks to Gronwall's inequality and Lemma \ref{lem-Uni-1}, we get
\begin{equation*}\label{3-54}
\begin{split}
&\sup_{\tau\in [0, t]}(\varepsilon\|\nabla\varphi(\tau)\|^2_{L^2}+\|\varphi(\tau)\|^2_{L^2})+\int^t_0\|\Delta\varphi(\tau)\|^2_{L^2}\, d\tau
\\
&\leq C(\varepsilon\|\nabla\varphi(0)\|^2_{L^2}+\|\varphi(0)\|^2_{L^2}+\int^t_0\|u_t\|^2_{L^2}\,d\tau)e^{C\int^t_0(\|\phi'(\theta)\|^4_{L^4}+1)\,d\tau}
 \leq C_{T},
\end{split}
\end{equation*}
where $C_{T}$ is independent of $\varepsilon \in (0, 1)$, in particular,
\begin{equation}\label{3-54-1}
\|\partial_t\theta\|^2_{L^\infty([0, T]; L^2(\mathbb{T}^2))}\leq C_{T}.
\end{equation}
From this, similar to estimates \eqref{3-24} and \eqref{3-25}, we may improve \eqref{3-26} to
\begin{equation}\label{3-55}
\|\tilde{h}(t)\|_{L^{\infty}([0,T];L^\infty(\mathbb{T}^2))}\leq C_{T},
\end{equation}
where $C_{T}$ is independent of $\varepsilon\in (0, 1)$.

Indeed, note that
\begin{equation}\label{3-55-1}
\tilde{h}=m(\phi(\theta))-(-\Delta)^{-1}\partial_t\theta-(-\Delta)^{-1}((u\cdot\nabla)\theta),
\end{equation}
then we get from \eqref{3-54-1} that
\begin{equation*}\label{3-55-2}
\begin{split}
&\|(-\Delta)^{-1}\partial_t\theta(t)\|^2_{L^{\infty}([0,T];L^\infty(\mathbb{T}^2))}\\
& \leq \|(-\Delta)^{-1}\partial_t\theta\|_{L^{\infty}([0,T];L^2)}\|\partial_t\theta\|_{L^{\infty}([0,T];L^2)}\, d\tau\leq \|\partial_t\theta\|^2_{L^{\infty}([0,T];L^2)}\leq C_T,
\end{split}
\end{equation*}
and
\begin{equation*}\label{3-55-3}
\begin{split}
&\|(-\Delta)^{-1}((u\cdot\nabla)\theta)\|^2_{L^{\infty}([0,T];L^\infty(\mathbb{T}^2))}
\leq C\|(-\Delta)^{-1}\dive\,(u\theta)\|^2_{L^{\infty}([0,T];L^\infty)}
\\
&\leq C\|u\theta\|_{L^{\infty}([0,T]; L^4)}^2 \leq C\|u\|^2_{L^{\infty}([0,T]; L^{4})}\|\theta\|^2_{L^\infty([0,T] \times \mathbb{T}^2)}\leq C_T,
\end{split}
\end{equation*}
which along with \eqref{3-19-1} and \eqref{3-55-1} implies \eqref{3-55}.

Therefore, thanks to Proposition \ref{prop-3} and \eqref{3-55}, for two auxiliary ODEs:
\begin{equation*}
\begin{cases}
&\varepsilon y'_{\pm}+\phi(y_{\pm})=h_{\pm}, \quad \forall~t\geq0,
\\
&y_{\pm}(0)=\pm\|\theta_0\|_{L^\infty(\mathbb{T}^2)}
\end{cases}
\end{equation*}
with $h_{\pm}(t):=\|\tilde{h}(t)\|_{L^\infty(\mathbb{T}^2)}$ , one can get
\begin{equation*}\label{3-27-1}
\|y_{\pm}(t)\|_{L^{\infty}((0, T))}\leq 1-\delta_T
\end{equation*}
for some $\delta_T \in (0, 1)$, from which, it results in \eqref{3-31}, and also \eqref{3-47} holds true.
This ends the proof of Theorem \ref{thm-Uni-M}.
\end{proof}

Let's now claim that \eqref{3-47} in Theorem \ref{thm-Uni-M} holds for the solution $\theta$ of the original Cahn-Hilliard equation \eqref{3-1}.
\begin{thm}\label{thm-Uni-sep}
Let $\theta_0\in H^s$ with $s\geq 4$, $\int_{\mathbb{T}^2} \theta_0\, dx=0$, and $\|\theta_0\|_{L^\infty(\mathbb{T}^2)}\leq 1-\delta_0$ for $\delta_0\in(0,1)$, $0<T<+\infty$, $u\in L^{\infty}([0,T); H^1(\mathbb{T}^2))\cap L^2([0,T]; H^2(\mathbb{T}^2))$ with $\partial_t u\in L^2([0,T]; L^2(\mathbb{T}^2))$, $\dive\, u=0$,
and the function $\phi=\Phi'$ with $\Phi$ in \eqref{1-9}. Then there exists a solution $\theta$ to \eqref{3-1} on $[0,T)$ satisfying \eqref{3-56-2},
$\theta \in \mathbb{C}([0,T); \mathbb{D}_{0}))$, $\partial_t \theta \in L^2([0,T); H^{1}(\mathbb{T}^2))$, and also there is a positive constant $\delta=\delta(\delta_0, T)$ such that
\begin{equation}\label{c5-7-1}
\|\theta\|_{L^\infty([0, T)\times \mathbb{T}^2)}\leq1-\delta.
\end{equation}
Moreover, if there is another solution to $\bar{\theta}$ to \eqref{3-1} on $[0,T)$ satisfying \eqref{3-56-2},
$\bar{\theta}\in {C}([0,T); H^2))$, $\nabla\mu( \bar{\theta}) \in L^2([0,T); L^2(\mathbb{T}^2))$, then $\bar\theta\equiv\theta $ on $[0,T)\times \mathbb{T}^2$.
\end{thm}

\begin{proof}
Let $\{\theta^{\varepsilon}\}_{\varepsilon>0}$ be the solution sequence of the approximate equations \eqref{3-3-0} in Theorem \ref{thm-Uni-M}, then it follows from Lemma \ref{thm-uni-ch} and Theorem \ref{thm-Uni-M} that there exists a subsequence of $\{\theta^{\varepsilon}\}_{\varepsilon>0}$ , still denoted by $\{\theta^{\varepsilon}\}_{\varepsilon>0}$,  converges, as $\varepsilon$ goes to zero,
to some function ${\theta}$, defined on $[0, T) \times \mathbb{T}^2$ satisfying \eqref{c5-7-1}, $\int_{\mathbb{T}^2} \theta(t)\, dx=0$ for any $t\in [0, T) $, ${\theta} \in {C}([0,T); \mathbb{D}_{0}))$, $\partial_t {\theta} \in L^2([0,T); H^{1}(\mathbb{T}^2))$, and solving \eqref{3-1} in the weak sense. Let's now pay attention to the proof of the uniqueness. Denote $\widetilde{\theta}:=\bar\theta-\theta$, then $\widetilde{\theta}$ solves
\begin{equation} \label{c5-7-2}
\begin{cases}
&(-\Delta)^{-1}\bigg(\partial_t\widetilde{\theta}+ \dive\,(u\,\widetilde{\theta})\bigg)-\Delta \widetilde{\theta}=-\bigg(\phi(\bar\theta)-\phi(\theta)\bigg)+\bigg(m(\phi(\bar\theta))-m(\phi(\theta))\bigg),\\\
& \qquad\qquad\qquad\qquad\qquad\qquad\qquad\qquad\qquad\qquad\qquad\forall \, (t, x) \in (0,T) \times \mathbb{T}^2,
\\
&\widetilde{\theta}|_{t=0}=0.
\end{cases}
\end{equation}
Taking the $L^2$ inner product of \eqref{c5-7-2} with $\widetilde{\theta}$ yields
\begin{equation}\label{diff-con-1}
\begin{split}
&\frac{1}{2}\frac{d}{dt}\|\widetilde{\theta}\|_{\dot H^{-1}}^2+\int_{\mathbb{T}^2}\widetilde{\theta}\,(-\Delta)^{-1} \dive\,(u\,\widetilde{\theta})\,dx+\|\widetilde{\theta}\|_{\dot H^{1}}^2+\int_{\mathbb{T}^2}(\phi(\bar\theta)-\phi(\theta))\,\widetilde{\theta}\, dx=0,
\end{split}
\end{equation}
where we have used the facts that $m(\widetilde{\theta})=0$.

Since $\phi' \geq -\alpha$, we get
\begin{equation*}
\int_{\mathbb{T}^2}(\phi(\bar\theta)-\phi(\theta))\,\widetilde{\theta}\, dx \geq -\alpha\|\widetilde{\theta}\|_{L^2}^2,
\end{equation*}
from which, we infer from \eqref{diff-con-1} that
\begin{equation*}
\begin{split}
&\frac{1}{2}\frac{d}{dt}\|\widetilde{\theta}\|_{\dot H^{-1}}^2+\|\widetilde{\theta}\|_{\dot H^{1}}^2\leq \alpha\|\widetilde{\theta}\|_{L^2}^2+\|\widetilde{\theta}\|_{\dot H^{-1}}\|u\|_{L^{\infty}} \|\widetilde{\theta}\|_{L^2} \\
&\leq \frac{1}{2}\|\nabla \widetilde{\theta}\|_{L^2}^2+C\|\widetilde{\theta}\|_{\dot H^{-1}}^2(1+\|u\|_{L^{2}}\|\Delta u\|_{L^{2}}).
\end{split}
\end{equation*}
Hence, it follows
\begin{equation*}
\begin{split}
&\frac{d}{dt}\|\widetilde{\theta}\|_{\dot H^{-1}}^2+\|\widetilde{\theta}\|_{\dot H^{1}}^2 \leq C\|\widetilde{\theta}\|_{\dot H^{-1}}^2(1+\|u\|_{L^{2}}\|\Delta u\|_{L^{2}}).
\end{split}
\end{equation*}
Gronwall's inequality gives rise to
\begin{equation*}
\begin{split}
&\sup_{t \in [0, T]}\|\widetilde{\theta}(t)\|_{\dot H^{-1}} =0,
\end{split}
\end{equation*}
which implies $\widetilde{\theta}\equiv 0$ on $[0,T)\times \mathbb{T}^2$, and then ends the proof of Theorem \ref{thm-Uni-sep}.
\end{proof}

\renewcommand{\theequation}{\thesection.\arabic{equation}}
\setcounter{equation}{0}
\section{Global well-posedness }\label{sect-6}

In this section, we prove the global well-posedness of the Navier-Stokes-Cahn-Hilliard system \eqref{NSCH-2}.

Towards this, let's first study the global $H^{s_0}$-estimate of $\theta$ for $s_0\in (1, \frac{3}{2}]$.
\begin{lem}\label{lem-theta-Hs-1}
Under the assumptions in Theorem \ref{thm-lp-nsch}, if, in addition, $s \geq 4$, then there holds
\begin{equation}\label{8-1}
\begin{split}
&\|\theta\|^2_{L^\infty([0,T); H^{s_0}(\mathbb{T}^2))}+\|\theta\|^2_{L^2((0,T); H^{2+s_0}(\mathbb{T}^2))}\leq C_T(u_0,\theta_0)
\end{split}
\end{equation}
for any $s_0\in (1, \frac{3}{2}]$.
\end{lem}
\begin{proof}
Firstly, similar to the proof of \eqref{3-6}, we may deduce, from the basic energy estimate \eqref{3-56-3}, that
\begin{equation}\label{8-1-1}
\begin{split}
\sup_{\tau \in [0, T]}\|\theta(\tau)\|^2_{\mathbb{D}_{0}}+\int^t_0\|\partial_t\theta(\tau)\|^2_{H^1(\mathbb{T}^2)}\, d\tau\leq C_T,
\end{split}
\end{equation}
which follows from Theorem \ref{thm-Uni-sep} that
\begin{equation}\label{8-2}
\|\theta\|_{L^{\infty}([0, T) \times \mathbb{T}^2)} < 1-\delta_T
\end{equation}
for some $\delta_T>0$.

Hence, from \eqref{3-56-3} and \eqref{8-2}, we get
\begin{equation}\label{global-4-2}
\begin{split}
\|\nabla\Delta\theta\|_{L^2([0, T]; L^2)} &\leq \|\nabla\phi(\theta)\|_{L^2([0, T]; L^2)}+\|\nabla\mu\|_{L^2([0, T]; L^2)}\\
&\leq C\|\nabla\theta\|_{L^2([0, T]; L^2)}+\|\nabla\mu\|_{L^2([0, T]; L^2)}<C_T.
\end{split}
\end{equation}
Taking the $H^{s_1}(\mathbb{T}^2)$ (with $s_1>1$) inner product of the second equation of \eqref{NSCH-2} with $\theta$, we get
\begin{equation}\label{global-3-87}
\begin{split}
&\frac{1}{2}\frac{d}{dt}\|\theta\|^2_{H^{s_1}}+\|\Delta\theta\|^2_{H^{s_1}}=\int_{\mathbb{T}^2}   \Lambda^{s_1}(u\, \theta)\cdot\nabla\Lambda^{s_1}\theta\, dx+\int_{\mathbb{T}^2}  \Lambda^{s_1}\phi(\theta)\, \Lambda^{s_1}\Delta\theta\, dx.
\end{split}
\end{equation}
From \eqref{8-2} and Lemma \ref{lem-4}, it follows
\begin{equation}\label{global-3-87-1}
\begin{split}
&|\int_{\mathbb{T}^2}  \Lambda^{s_1}\phi(\theta)\,\Lambda^{s_1}\Delta\theta\, dx|\leq  \frac{1}{4}\|\Delta \theta\|_{H^{{s_1}}}^2+ C\|\phi(\theta)\|_{H^{s_1}}^2\leq  \frac{1}{4}\|\Delta\theta\|_{H^{{s_1}}}^2+ C \|\theta\|_{H^{s_1}}^2,
\end{split}
\end{equation}
and
\begin{equation}\label{global-3-87-2}
\begin{split}
&|\int_{\mathbb{T}^2}   \Lambda^{s_1}(u\, \theta)\cdot\nabla\Lambda^{s_1}\theta\, dx|\leq C(\|u\|_{H^{s_1}}\|\theta\|_{L^{\infty}}+\|u\|_{L^{\infty}}\|\theta\|_{H^{s_1}})\|\nabla\theta\|_{H^{s_1}}\\
&\leq  \frac{1}{4}\|\Delta \theta\|_{H^{{s_1}}}^2+ C\|\theta\|_{H^{s_1}}^2\|u\|_{L^{\infty}}^2+\|u\|_{H^{s_1}}^2.
\end{split}
\end{equation}
Plugging \eqref{global-3-87-1} and \eqref{global-3-87-2} into \eqref{global-3-87} yields
\begin{equation}\label{global-3-87-3}
\begin{split}
&\frac{d}{dt}\|\theta\|_{H^{{s_1}}}^2+\|\Delta\theta\|_{H^{{s_1}}}^2 \leq C(\|u\|_{H^{{s_1}}}^2+\|\theta\|_{H^{{s_1}}}^2)(1+\|u \|_{L^{\infty}}^2).
\end{split}
\end{equation}
Taking ${s_1}=s_0 \in (1, \frac{3}{2}]$ in \eqref{global-3-87-3}, it follows from Gronwall's inequality that
\begin{equation}\label{global-3-87-4}
\begin{split}
&\sup_{\tau\in [0, t]}\|\theta(\tau)\|_{H^{s_0}}^2+\int_0^t\|\Delta\theta\|_{H^{s_0}}^2\, d\tau\\
&\leq C(\|\theta_0\|_{H^{s_0}}^2+\int_0^t\|u\|_{H^{s_0}}^2(1+\|u \|_{L^{\infty}}^2)\, d\tau) e^{C\int_0^t(1+\|u \|_{L^{\infty}}^2)\, d\tau}\\
&\leq C(t)\bigg(\|\theta_0\|_{H^{s_0}}^2+\int_0^t\|u\|_{H^{2(s_0-1)}}\|u\|_{H^{2}}(1+\| u\|_{L^{2}}\|u\|_{H^2})\, d\tau\bigg)\\
&\qquad\qquad\qquad\qquad\qquad\qquad\qquad\qquad\qquad\times e^{C\int_0^t(1+\| u\|_{L^{2}}\|u\|_{H^2})\, d\tau}.
\end{split}
\end{equation}
Therefore, combining \eqref{global-3-87-4} with \eqref{3-56-3} and \eqref{global-4-2}, we deduce \eqref{8-1}.
\end{proof}

We are now in a position to complete the proof of Theorem \ref{thm-glo-wp}.
\begin{proof}[Proof of Theorem \ref{thm-glo-wp}]

Thanks to Theorem \ref{thm-lp-nsch}, we conclude that: under the assumptions in Theorem \ref{thm-glo-wp}, \eqref{NSCH-2} has a
unique local solution $(u, \, \theta)$ satisfying \eqref{local-space} and \eqref{separate-cond-local-1}.
Assume that ${T}^\ast>0$ is the maximal existence time of this solution, that is
\begin{equation*}\label{local-space-1}
\begin{split}
(u, \theta)\in &\bigg(C([0, T^\ast); H^s(\mathbb{T}^2))\cap (L^2_{loc}([0,T^\ast); H^{s+1}(\mathbb{T}^2)) \bigg) \\
&\qquad \times \bigg(C([0,T^\ast); H^s(\mathbb{T}^2))\cap L^2_{loc}([0,T^\ast); H^{s+2}(\mathbb{T}^2)))\bigg).
\end{split}
\end{equation*}
It suffices to prove ${T}^\ast=+\infty$. We will argue by contradiction argument. Hence, we assume ${T}^\ast<+\infty$ in what follows.

 Since the system \eqref{NSCH-2} has a smoothing effect to the solution $(u, \theta)$, we may assume, without loss of generality, the regularity index $s>4$ of the initial data in Theorem \ref{thm-lp-nsch} according to \eqref{local-space}.

Thanks to Theorem \ref{thm-Uni-M}, we find that, for every $t\in[0,T^\ast)$, there exists a positive constant $\delta_1 \in (0, 1)$ such that
\begin{equation}\label{7-2}
\|\theta\|_{L^\infty([0, T^\ast)\times \mathbb{T}^2)}\leq  1-\delta_1.
\end{equation}
Taking the $H^s(\mathbb{T}^2)$ inner product of the first equation of \eqref{NSCH-2} with $u$, we have
\begin{equation}\label{7-3}
\begin{split}
\frac{1}{2}\frac{d}{dt} \|u\|_{H^s}^2 &+\int_{\mathbb{T}^2}   2\nu(\theta) \Lambda^s D(u) : \Lambda^s \nabla u \,dx=\int_{\mathbb{T}^2} ([\Lambda^s, u]\cdot \nabla u):\Lambda^s \nabla u \,dx  \\
&-\int_{\mathbb{T}^2}  [\Lambda^s, 2\nu(\theta)] D(u): \Lambda^s \nabla u \,dx+\int_{\mathbb{T}^2}  \Lambda^s (\mu\nabla\theta)\cdot \Lambda^s u \,dx=:\sum_{i=1}^3 I_i.
\end{split}
\end{equation}
We first check from Lemmas \ref{lemma-comm} and \ref{lem-4} that
\begin{equation}\label{7-4}
\begin{split}
|I_1|&=|\int_{\mathbb{T}^2} ([\Lambda^s, u]\cdot \nabla u):\Lambda^s \nabla u \,dx |\leq C\|\nabla u\|_{L^{\infty}}\|u\|_{H^s}\|\nabla u\|_{H^s}\\
&\leq \eta\|\nabla u\|_{H^s}^2+C_{\eta}\|\nabla u\|_{L^{\infty}}^2\|u\|_{H^s}^2,
\end{split}
\end{equation}
and
\begin{equation}\label{7-5}
\begin{split}
|I_2|&=|\int_{\mathbb{T}^2}  [\Lambda^s, 2\nu(\theta)] D(u): \Lambda^s \nabla u \,dx|
\\
&\leq C \|\nabla u\|_{H^s}(\|\nu(\theta)-\nu(0)\|_{H^s}\|\nabla u\|_{L^{\infty}}+\|\nabla(\nu(\theta)-\nu(0))\|_{L^{\infty}}\|\nabla u\|_{H^{s-1}})\\
&\leq C \|\nabla u\|_{H^s}(\|\theta\|_{H^s}\|\nabla u\|_{L^{\infty}}+\|\nabla\theta\|_{L^{\infty}}\|u\|_{H^{s}})\\
&\leq \eta\|\nabla u\|_{H^s}^2+C_{\eta}(\|\nabla u\|_{L^{\infty}}^2+\|\nabla\theta\|_{L^{\infty}}^2)(\|u\|_{H^s}^2+\|\theta\|_{H^s}^2)
\end{split}
\end{equation}
for any positive constant $\eta$. While for $I_3$, we deduce from \eqref{7-2} and Lemma \ref{lem-4} that
and
\begin{equation*}
\begin{split}
|I_3|&=|\int_{\mathbb{T}^2}  \Lambda^s (\mu\nabla\theta)\cdot \Lambda^s u \,dx|
\\
&\leq C \|u\|_{H^s}\bigg((\|\phi(\theta)\|_{H^s}+\|\Delta\theta\|_{H^s})\|\nabla \theta\|_{L^{\infty}}+(\|\phi(\theta)\|_{L^{\infty}}+\|\Delta\theta\|_{L^{\infty}})\|\nabla \theta\|_{H^s}\bigg)\\
&\leq C \|u\|_{H^s}\bigg((\|\theta\|_{H^s}+\|\Delta\theta\|_{H^s})\|\nabla \theta\|_{L^{\infty}}+(1+\|\Delta\theta\|_{L^{\infty}})\|\nabla \theta\|_{H^s}\bigg),
\end{split}
\end{equation*}
which follows that
\begin{equation}\label{7-6}
\begin{split}
&|I_3|
\leq \eta \|\Delta\theta\|_{H^s}^2+C_{\eta}(1+\|\nabla\theta\|_{L^{\infty}}^2+\|\Delta\theta\|_{L^{\infty}}^2)(\|u\|_{H^s}^2+\|\theta\|_{H^s}^2).
\end{split}
\end{equation}
On the other hand, we get, for some positive $c$,
\begin{equation}\label{7-5-1}
\begin{split}
&\int_{\mathbb{T}^2}   2\nu(\theta) \Lambda^s D(u) : \Lambda^s \nabla u \,dx =\int_{\mathbb{T}^2}  2\nu(\theta) |\Lambda^s D(u)|^2\, dx \geq c\nu_1\|\nabla u\|_{H^s}^2,
\end{split}
\end{equation}
where $\nu_1>0$ is a lower bound of $\nu(\cdot)$ on $[-1, 1]$.

Therefore, substituting \eqref{7-4}-\eqref{7-5-1} into \eqref{7-3} we deduce that
\begin{equation}\label{7-7}
\begin{split}
&\frac{d}{dt} \|u\|_{H^s}^2+\frac{3}{2} c\nu_1\|\nabla u\|_{H^{s}}^2\\
&\leq \eta \|\Delta\theta\|_{H^s}^2+ C_{\eta}(1+\|\Delta\theta\|_{L^{\infty}}^2+\|\nabla u\|_{L^{\infty}}^2+\|\nabla\theta\|_{L^{\infty}}^2)(\|u\|_{H^s}^2+\|\theta\|_{H^s}^2).
\end{split}
\end{equation}
Taking $\eta$ in \eqref{7-7} small enough and combining \eqref{global-3-87-3} with \eqref{7-7}, we infer
\begin{equation}\label{7-11}
\begin{split}
&\frac{d}{dt} (\|u\|_{H^s}^2+\|\theta\|_{H^s}^2)+ c\nu_1\|\nabla u\|_{H^{s}}^2+\|\Delta\theta\|_{H^s}^2\\
&\leq C(\|u\|_{H^s}^2+\|\theta\|_{H^s}^2)(1+\|\Delta\theta\|_{L^{\infty}}^2+\| u\|_{L^{\infty}}^2+\|\nabla\theta\|_{L^{\infty}}^2+\|\nabla u\|_{L^{\infty}}^2),
\end{split}
\end{equation}
which follows that
\begin{equation}\label{7-12-1}
\begin{split}
&\sup_{\tau\in [0, t]}(\|u(\tau)\|_{H^s}^2+\|\theta(\tau)\|_{H^s}^2)\\
&\leq (\|u_0\|_{H^s}^2+\|\theta_0\|_{H^s}^2)e^{C\int_0^t(1+\|(u, \,\nabla\theta, \,\Delta\theta)\|_{L^\infty}^2+\|\nabla u\|_{L^{\infty}}^2)\, d\tau}.
\end{split}
\end{equation}
Thanks to \eqref{3-56-3}, \eqref{8-1}, and the Sobolev embedding $H^s(\mathbb{T}^2) \hookrightarrow L^{\infty}(\mathbb{T}^2) $ with $s>1$, we infer from \eqref{7-12-1} that
\begin{equation}\label{7-12}
\begin{split}
&\sup_{\tau\in [0, t]}(\|u(\tau)\|_{H^s}^2+\|\theta(\tau)\|_{H^s}^2)\leq C(\|u_0\|_{H^s}, \|\theta_0\|_{H^s}, t)\exp\{C\int_0^t\|\nabla u\|_{L^\infty}^2\, d\tau\}.
\end{split}
\end{equation}
Applying Lemma \ref{lem-logarithmic} to \eqref{7-12} yields
\begin{equation*}\label{7-13}
\begin{split}
&(e+\|u(t)\|_{H^s}^2+\|\theta(t)\|_{H^s}^2)\leq C(t)\exp\{C\int_0^t(1+\|u\|_{H^2}^2)\log(e+\|u\|_{H^s}^2)\, d\tau\},
\end{split}
\end{equation*}
where we have used the fact $s>2$, which leads to
\begin{equation*}\label{7-14}
\begin{split}
&\log (e+\|u(t)\|_{H^s}^2+\|\theta(t)\|_{H^s}^2)\leq C(t)+ C \int_0^t(1+\|u\|_{H^2}^2)\log(e+\|u\|_{H^s}^2)\, d\tau.
\end{split}
\end{equation*}
Therefore, we get from Gronwall's inequality and \eqref{3-56-3} that
\begin{equation*}\label{7-15}
\begin{split}
&\log (e+\|u(t)\|_{H^s}^2+\|\theta(t)\|_{H^s}^2)\\
&\leq (\log (e+\|u_0\|_{H^s}^2+\|\theta_0\|_{H^s}^2)+C(t))\exp\{ C\int_0^t(1+\|u\|_{H^2}^2)\, d\tau\} \leq C(t),
\end{split}
\end{equation*}
which follows that
\begin{equation*}
\begin{split}
& \sup_{\tau\in [0, T^\ast)}(\|u(\tau)\|_{H^s}^2+\|\theta(\tau)\|_{H^s}^2) \leq C_{T^\ast} <+\infty.
\end{split}
\end{equation*}
From this, the solution can be extended after $t=T^\ast$, which contradicts with the definition of $T^\ast$. Hence, we get $T^\ast=+\infty$, and then complete the proof of Theorem \ref{thm-glo-wp}.
\end{proof}

\vskip 0.2cm

\noindent {\bf Acknowledgments.} The work of Guilong Gui is supported in part by NSF of China under Grant 11571279 and 11331005. Zhenbang Li is supported in part by NSF of China under Grant 11801443 and SXDE Fund 15JK1347.

\vskip 0.2cm

\renewcommand{\theequation}{\thesection.\arabic{equation}}
\setcounter{equation}{0}

\appendix

\section{Appendix}

The proof  of Theorem \ref{thm-lp-nsch} requires a dyadic decomposition
of the Fourier variables, or Littlewood-Paley decomposition, which may be explained how it may be built in the case $x\in\mathbb{R}^d$ or $\mathbb{T}^d$ (see e.g. \cite{BCD, Ch1, Tri}) as follows..

Let us first recall a dyadic partition of  unity. We define by  $\mathcal{C}$ the ring of center~$0$, of small radius~$3/4$
and  great radius~$8/3$. Then it exists two radial functions~$\chi$ and~$\varphi$ the  values of
which are in the
interval~$[0,1]$, belonging respectively  to~$\cD(B(0,4/3))$ and
to~$\mathcal{D}(\mathcal{C})$ such that
\begin{equation*}
\label{inho-LP}
\chi(\xi)  + \sum_{j\geq 0} \varphi (2^{-j}\xi) =  1 \quad (\forall \xi\in\mathbb{R}^d), \quad \sum_{j\in \mathbb{Z}} \varphi (2^{-j}\xi) =  1 \quad (\forall  \xi\in \mathbb{R}^d\setminus\{0\}),
\end{equation*}
\begin{equation*}
\label{lpfond2}
|j-j'|\geq  2
\Rightarrow
\Supp \varphi(2^{-j}\cdot)\cap \Supp  \varphi(2^{-j'}\cdot)=\emptyset,
\end{equation*}
\begin{equation*}
\label{lpfond3}
j\geq  1
\Rightarrow
\Supp \chi\cap \Supp \varphi(2^{-j}\cdot) =  \emptyset.
\end{equation*}
If $\widetilde{\mathcal{C}}= B(0,2/3)+{\mathcal{C}}$, then $\widetilde{\mathcal{C}}$
is a ring and we have
$
|j-j'|\geq  5\Rightarrow
2^{j'}{\widetilde{\mathcal{C}}}
\cap 2^j {\mathcal C} =  \emptyset,
$ and
\begin{equation*}
\label{lpfond5}
\frac1 3 \leq \chi^2(\xi)
+ \sum_{j\geq 0} \varphi^2(2^{-j}\xi)  \leq 1 \quad (\forall  \xi\in \mathbb{R}^d),
\quad
\frac 1 2\leq \sum_{j\in \mathbb{Z}} \varphi^2(2^{-j}\xi)  \leq 1 \quad (\forall  \xi\in \mathbb{R}^d\setminus\{0\}).
\end{equation*}
From now on, we fix two functions $\chi$ and $\varphi$ satisfying
the above assertions and
denote $h\eqdefa{\mathcal F}^{-1}\varphi,$
$\widetilde{h}\eqdefa{\mathcal F}^{-1}\chi$.
The inhomogeneous dyadic blocks ${\Delta}_j$ and the inhomogeneous
low-frequency cut-off operator ${S}_j$ are defined for all $ j \in
\mathbb{N}\cup \{-1\}$ by
\begin{equation*}\label{cam-2.01-a}
\begin{split}
&{\Delta}_{j}
f\eqdefa\varphi(2^{-j}D)f=2^{j d}\int_{\mathbb{R}^{d}}h(2^j y)f(x-y)dy, \quad \forall \quad j \geq 0,\\
&{\Delta}_{-1}
f\eqdefa\chi(D)f=\int_{\mathbb{R}^{d}}\widetilde{h}( y)f(x-y)dy,\quad {\Delta}_{j}
f\eqdefa 0\quad (\forall \, j \leq -2),\\
&{S}_{j} f\eqdefa\sum_{ j^{\prime} \le j-1}{\Delta}_{j^{\prime}}
f=\chi(2^{-j}D)f=2^{j d}\int_{\mathbb{R}^{d}}\widetilde{h}(2^j
y)f(x-y)dy.
\end{split}
\end{equation*}

We should point out that all the above operators~$\Delta_j$ and~$S_{j}$
maps $L^p$ into $L^p$ with norms  which do not depend on~$j$. This
fact will be used all along this paper.

With above notations in hand, the inhomogenous Sobolev
space $H^s(\mathbb{R}^{d})$ can be equivalently defined by
\begin{equation*}
H^s(\mathbb{R}^{d})\eqdefa\{f\in {\mathcal
S}^{\prime}(\mathbb{R}^{d})\,: \,  \|f\|_{H^s}<\infty\}\quad \mbox{with}\quad
\|f\|_{H^s}\eqdefa \bigg\|\bigg(2^{j s }\|\Delta_j
f\|_{L^2}\bigg)_{j \in \mathbb{N} \cup \{-1\}} \bigg\|_{\ell^{2}}
\end{equation*}

\begin{rmk}\label{rmk-cha-1}

Let $s\in \mathbb{R}$, then, $u$ belongs to ${B}^{s}_{p, r}$ if and only if there exists
$\{c_{j}\}_{j \in \mathbb{N} \cup \{-1\}} $ such that $\|c_{j}\|_{\ell^{2}} =1$
and
\begin{equation*}
\|{\Delta}_{j}u\|_{L^{2}}\leq C c_{j} 2^{-j s } \|u\|_{H^s}.
\end{equation*}
\end{rmk}
\begin{lem}\label{cpam-lem2.1}(\cite{BCD}, Bernstein-type lemma) Let $\mathcal{B}$ be a ball and $\mathcal{C}$ a ring of $\mathbb{R}^d.$
 A constant $C$ exists so
that for any positive real number $\lambda$, any non negative
integer $k$, any  homogeneous function $\sigma$ of degree $m$ smooth
outside of $0$, and any couple of real numbers $(a, \; b)$ with $ b
\geq a \geq 1,$ there hold
\begin{equation*}
\begin{split}
&\Supp \hat{u} \subset \lambda \mathcal{B} \Rightarrow
\sup_{|\alpha|=k} \|\pa^{\alpha} u\|_{L^{b}} \leq  C^{k+1}
\lambda^{k+ d(\frac{1}{a}-\frac{1}{b} )}\|u\|_{L^{a}},\\
& \Supp \hat{u} \subset \lambda \mathcal{C} \Rightarrow
C^{-1-k}\lambda^{ k}\|u\|_{L^{a}}\leq
\sup_{|\alpha|=k}\|\partial^{\alpha} u\|_{L^{a}} \leq
C^{1+k}\lambda^{ k}\|u\|_{L^{a}},\\
& \Supp \hat{u} \subset \lambda \mathcal{C} \Rightarrow \|\sigma(D)
u\|_{L^{b}}\leq C_{\sigma, m} \lambda^{ m+d(\frac{1}{a}-\frac{1}{b}
)}\|u\|_{L^{a}}.
\end{split}
\end{equation*}
\end{lem}

In  order to obtain a better description of the regularizing effect
of the transport-diffusion equation, we will use Chemin-Lerner type
spaces $\widetilde{L}^{\lambda}_T(B^s_{p,r}(\mathbb{R}^d))$ from
\cite{Ch99, CL}. Let $s\in\mathbb{R}$,
$\lambda \in[1,\,+\infty]$ and $T\in(0,\,+\infty]$. We define
$\widetilde{L}^{\lambda}_T(H^s(\mathbb{R}^d))$ as the completion
of $C([0,T],\mathcal{S}(\mathbb{R}^d))$ by the norm
$
\| f\|_{\widetilde{L}^{\lambda}_T(H^s)} \eqdefa
\Big(\sum_{q\in\Z}2^{2qs} \|\Delta_q\,f
\|_{L^{\lambda}([0, T]; L^2)}^2 \Big)^{\frac{1}{2}}
<\infty.
$
Thanks to this definition, Minkowskii's inequality ensures that
\begin{equation*}\label{C-L-space-1}
\begin{split}
&\| f\|_{L^{\lambda}_T(H^s)} \leq \| f\|_{\widetilde{L}^{\lambda}_T(H^s)} \quad \mbox{if} \quad \lambda > 2,\\
&\| f\|_{\widetilde{L}^{\lambda}_T(H^s)} \leq \| f\|_{L^{\lambda}_T(H^s)} \quad \mbox{if} \quad \lambda < 2,\\
&\| f\|_{L^{\lambda}_T(H^s)} =\| f\|_{\widetilde{L}^{\lambda}_T(H^s)} \quad \mbox{if} \quad \lambda = 2.
\end{split}
\end{equation*}

In what follows, we shall frequently use  Bony's decomposition
\cite{Bo} in the inhomogeneous context:
\begin{equation}\label{bony}
\begin{split}
&uv={T}_u v+{R}(u,v)={T}_u v+{T}_v u+\mathcal{R}(u,v)
\end{split}
\end{equation}
where
\begin{equation*}
\begin{split}
&{T}_u v\eqdefa\sum_{q \in \mathbb{Z}}S_{q-1}u\Delta_q
v,\qquad
{R}(u,v)\eqdefa\sum_{q\in\mathbb{Z}}\Delta_q u  S_{q+2}v,\\
&\mathcal{R}(u,v)\eqdefa\sum_{q\in \mathbb{Z}}\Delta_q u
\widetilde{\Delta}_{q}v\quad \quad \mbox{and}\quad
\widetilde{\Delta}_{q}v\eqdefa \sum_{|q'-q|\leq
1}\Delta_{q'}v.
\end{split}
\end{equation*}

For the sake of completeness, we shall first recall the following
commutator's estimates which will be frequently used throughout the
succeeding sections.

The following basic lemma will be of constant use in this
paper.
\begin{lem}[Lemma 2.97 in \cite{BCD}]\label{lem-commutator}(Commutator estimates) Let $(p, q, r) \in [1, \infty]^3$, $\theta$ be a $C^1$ function on $\mathbb{R}^{d}$ such that $(1+|\cdot|) \hat{\theta} \in L^1$. There
exists a constant $C$ such that for any Lipschitz function a with gradient in $L^p$
and any function $b$ in $L^q$, we have, for any positive $\lambda$,
\begin{equation}\label{commutator-compact-0}
\|[\theta(\lambda^{-1} D), a]b\|_{L^r} \leq C \lambda^{-1}\|\grad a\|_{L^p}\|b\|_{L^q} \quad \mbox{with} \quad \frac{1}{p}+ \frac{1}{q}= \frac{1}{r}.
\end{equation}
\end{lem}

\begin{lem}\label{lem-pressure-1} \
Let $s >0$, $f \in  H^{s+2}(\mathbb{R}^2)$ and
$ \grad g\in H^{s}(\mathbb{R}^2)$. Then there holds
\begin{equation} \label{model-pre-unique-1}
\begin{split}
&\|[\Delta_{q}, f] \nabla g\|_{L^2}^2 \lesssim c_{q}2^{-q(s+1)}(\|f\|_{H^{s+2}}\|\nabla g\|_{L^2}+\|f\|_{H^{2}}\|\nabla g\|_{H^{s}}),
\end{split}
\end{equation}
where $\sum_{q \geq -1} c_q^2\leq 1$.
\end{lem}
\begin{proof} Thanks to \eqref{bony},
we get by using a standard commutator argument that
\begin{equation}\label{decomposition-commu-1}
[\Delta_q, f]\grad g=[\Delta_q,{T}_f]\nabla f +\Delta_q
{T}_{\grad g}f+\Delta_q {\mathcal{R}}(f,\grad g)- {R}(f,
\Delta_q g).
\end{equation}
Note that  $\|S_{k-1} \grad f\|_{L^{\infty}}
\lesssim \|f\|_{H^{2}},$ one gets from \eqref{commutator-compact-0} that
\begin{equation*}
\begin{split}
\| [\Delta_q,T_f]\grad g\|_{L^2} \lesssim & 2^{-q}\sum_{|
k-q|\leq 4} \|S_{k-1}\nabla
f\|_{L^\infty}\|\nabla\Delta_{k}g\|_{L^2} \lesssim
c_{q}2^{-(s+1)q} \|f\|_{H^{2}}\| \grad g\|_{
H^{s}}.
\end{split}
\end{equation*}
While $\|S_{k-1} \grad g\|_{L^{\infty}}
\lesssim c_{k} 2^{k}\| \grad g\|_{L^2},$ which leads to
\begin{equation*}\label{model-pre-1-3-non-a}
\begin{split}
\|\Delta_q T_{\nabla g}f\|_{L^2} \lesssim \sum_{| q-k|\leq
4}\|S_{k-1} \grad g\|_{L^\infty}\|\Delta_k f\|_{L^2} \lesssim c_{q}^2
2^{-(s+1)q} \|f\|_{H^{s+2}}\| \grad g\|_{L^2}.
\end{split}
\end{equation*}
And applying  Lemma \ref{cpam-lem2.1} yields
\begin{equation*}
\begin{split}
\|\Delta_q \mathcal{R}(f, \nabla g)\|_{L^2} &\lesssim
2^{q}\sum_{k\geq q-3} \|\Delta_k f\|_{L^2}
\|\widetilde{\Delta}_k\nabla g\|_{L^2}\\
& \lesssim 2^{q}
\|f\|_{H^{s+2}}\| \grad g\|_{L^2} \sum_{k\geq q-3}c_{k}^2 2^{-(s+2)k} \lesssim c_{q}^2
2^{-(s+1)q} \|f\|_{H^{s+2}}\| \grad g\|_{L^2}.
\end{split}
\end{equation*}
The same estimate holds for $\| R(f,\Delta_q\nabla g)\|_{L^2}$.
Substituting  the above estimates into \eqref{decomposition-commu-1}, we
conclude the proof of \eqref{model-pre-unique-1}.
\end{proof}

\begin{rmk}\label{periodic-LP-1}
It is worth pointing out that all the properties described as above remain true in the periodic setting provided the
dyadic blocks have been defined as in \cite{DR}.
\end{rmk}

\end{document}